\documentclass[11pt, reqno]{amsart}
\usepackage{mathtools, amsthm}
\usepackage{xcolor}
\usepackage{amssymb, latexsym, amsmath, amsfonts, tikz, bm}
\usepackage{mathtools}
\usepackage{todonotes}
\usepackage{tikz-cd}
\usepackage{relsize}
\usepackage{mathabx}

\usepackage{marginnote}

\usepackage{soul}
\setstcolor{red}

\numberwithin{equation}{section}

\usepackage[inline]{enumitem}

\newtheorem{theorem}{Theorem}[section]
\newtheorem{proposition}[theorem]{Proposition}
\newtheorem{fact}[theorem]{Fact}
\newtheorem{lemma}[theorem]{Lemma}
\newtheorem{definition}[theorem]{Definition}
\newtheorem{ex}[theorem]{Example}
\newtheorem{remark}[theorem]{Remark}
\newtheorem{corollary}[theorem]{Corollary}

\newtheorem{conjecture}[theorem]{\bf Conjecture}

\usepackage[margin=1in, includeheadfoot]{geometry}

\theoremstyle{plain}

\title[Grothendieck and tensor product norms]{On a variant of the Grothendieck inequality\\ and estimates on tensor product norms}

\author[Gupta]{Rajeev Gupta}
\address{School of Mathematics and Computer Science, Indian Institute of Technology Goa, Goa - 403401}
\email{rajeev@iitgoa.ac.in}

\author[Misra]{Gadadhar Misra}
\address{Indian Statistical Institute, Bangalore and Indian Institute of Technology, Gandhinagar}
\email{gm@isibang.ac.in}

\author[Ray]{Samya Kumar Ray}
\address{The Institute of Mathematical Sciences, 4th Cross Street, CIT Campus, Tharamani, Chennai, Tamil Nadu 600113, India, and Homi Bhabha National Institute, Training School Complex, Anushakti Nagar, Mumbai 400094, India 
}
\email{samyaray7777@gmail.com, samya@imsc.res.in}

\thanks{The first named author was supported through the INSPIRE faculty grant (Ref. No. DST/INSPIRE/04/2017/002367). The second named author gratefully acknowledges the financial support provided by  SERB in the form of a ~J~C~Bose National Fellowship. The third named author acknowledges the DST-INSPIRE Faculty Fellowship DST/INSPIRE/04/2020/001132 and Prime Minister Early Career Research Grant Scheme ANRF/ECRG/2024/000699/PMS.}


\begin{document}

	\begin{abstract}
We investigate a Grothendieck-type inequality for pairs of Banach spaces $E,F$ assuming $E$ is finite-dimensional and study the associated Grothendieck-type constant. We prove that if there is a $C >0$ such that  $\|A\otimes \operatorname{id}_{F}\|_{E_m\check{\otimes}F\to E_n^*\hat{\otimes}F}\leqslant C \|A\|_{E_m\to E_n^*}$ for all $m,n\in\mathbb{N},$ where $\dim E_n=n$, then both $F$ and $F^*$ must have finite cotype. Moreover, assuming that $F$ has the bounded approximation property and a conjecture of Pisier has an affirmative answer, we show that $(E_n^*)_{n\geqslant 1}$ satisfies G.T. uniformly. We show that the Grothendieck-type constant defined for a pair of Banach spaces $(E,F)$ is closely related to another interesting quantity introduced recently by Auburn et al. comparing the projective and injective norms on the tensor product of two finite-dimensional Banach spaces $E$ and $F$. We also study  analogously the constants appearing in these extremal problems by restricting only to non-negative tensors. For contractive \emph{little} Parrott homomorphisms $\varrho_V : H^\infty(\Omega) \to M_{n}$, where $\Omega$ is the dual unit ball of a finite dimensional Banach space $(E,\|\cdot\|)$. We prove the sharp estimate $\|\varrho_V\|_{\mathrm{cb}}\leq\sqrt{\gamma(E)},$ $\gamma(E)$ being the positive Grothendieck constant associated with the pair $(E, \ell^n_2)$. This yields a new proof of a result of Davidson and Choi, using the lower bound $K_G^+(\ell_\infty^4,\ell_2^2) \geq 1.1658$ obtained in this paper. 
\end{abstract}

\maketitle
\section{Introduction}
In \cite{Grothendieck}, Grothendieck proved a remarkable theorem, which he himself called the ``Fundamental theorem of metric theory of tensor products", is now referred to as Grothendieck's theorem (in short G.T.). It has been a useful tool in several applications in the geometry of Banach spaces, operator theory and operator algebras, harmonic analysis, theoretical computer science, quantum information theory, and other fields. The reader may consult the recent survey article \cite{Pisier} for further information on this topic. 
Perhaps the simplest equivalent formulation of G.T. is by Lindenstrauss and Pe\l{}czyński \cite{LinPel} saying: There is a universal constant $K_G$, called the Grothendieck constant, such that 
\begin{equation}\label{GI}
\sup\bigg\{\Big |\sum_{i=1}^m\sum_{j=1}^na_{ij}\langle v_i,w_j\rangle \Big|:\|v_i\|_2=\|w_j\|_2=1\bigg\}\leqslant K_G,
\end{equation}
where $\langle \cdot,\cdot \rangle$ and $\|\cdot\|_2$ denote the inner product and norm respectively, of $\ell^\mathbb R_2(\mathbb N),$ and the supremum is taken over every real $m\times n$ 
 matrix $A=(a_{ij})_{i=1,j=1}^{m,n}$, $m,n\in\mathbb N$,  with \[\sup\bigg\{\Big |\sum_{i=1}^m\sum_{j=1}^n a_{ij}s_i{t}_j \Big |:|s_i|,|t_j| \leq 1, \bigg\}\leqslant 1.\] 

Inequality \eqref{GI} is called the Grothendieck inequality. A very similar statement appears in \cite{HandbookI}). We refer to Section \ref{prel} for any unexplained notation. Another equivalent formulation of the Grothendieck inequality, among many others,
is in \cite{RKS}: For any $m,n\in\mathbb N$ and $A=(a_{ij})_{1\leqslant i\leqslant m,1\leqslant j\leqslant n},$ there exists a positive constant $K$ such that 
\begin{equation} \label{msc} 
\|A\otimes \text{id}_{\ell_2}\|_{\ell_\infty^n\check{\otimes}\ell_2\to \ell_1^m\hat{\otimes}\ell_2}\leqslant K \|A\|_{\ell_\infty^n\to\ell_1^m}, \quad \forall A\in B(\ell^n_\infty, \ell_1^m),
\end{equation}
where $\text{id}_{\ell_2}$ denotes the identity operator $\text{id}:{\ell_2}\to \ell_2$ and $ B(X, Y)$ denotes the set of all linear maps from a normed linear space   $X$ to a normed linear space $Y$.
Moreover, the infimum of all admissible constants $K$ in \eqref{msc} coincides with the Grothendieck constant $K_G$. To prove this version, note that for any $m\times n$ scalar matrix $A=(a_{ij})_{1\leqslant i\leqslant m,1\leqslant j\leqslant n}$, we have \[\|A\|_{\ell_\infty^n\to\ell_1^m}=\sup\bigg\{\Big|\sum_{i=1}^m\sum_{j=1}^na_{ij}s_it_j\Big|:|s_i|=|t_j|=1\bigg\}.\]
Since $\ell_\infty^n\check{\otimes}\ell_2\cong\ell_\infty^n(\ell_2)$, it follows that $\|\sum_{i=1}^n\mathbf{e}_i\otimes x_i\|_{\ell_\infty^n\check{\otimes}\ell_2}=\max_{1\leqslant i\leqslant n}\|x_i\|_2$, where $\mathbf{e}_i$, $1\leqslant i \leqslant n$, is the standard basis of $\mathbb R^n$. Finally, by duality, we have $\|\sum_{j=1}^m\mathbf{e}_j\otimes x_j\|_{\ell_1^m\hat{\otimes}\ell_2}=\sum_{j=1}^m\|x_j\|_2$. The equality 
\begin{multline*} \sup \bigg\{\sum_{j=1}^n\Big \|\sum_{i=1}^ma_{ij}v_i \Big\|_2:\|v_i\|_2=1\bigg\}= \sup\bigg\{\Big |\sum_{i=1}^m\sum_{j=1}^na_{ij}\langle v_i,w_j\rangle \Big |:\|v_i\|_2=\|w_j\|_2=1\bigg\}\end{multline*}
is easy to verify, and the proof of the equivalence of the inequality in \eqref{msc} with Grothendieck inequality follows from it.  

In this article, we discuss a generalization of the Grothendieck inequality which is prompted by the equivalent form of the inequality we have just verified. All Banach spaces are assumed to be real if not mentioned otherwise. However, most of the results below also make sense for Banach spaces over the complex field. 
These can be proved with little or no change in the corresponding proof for the real case. 
We define the Grothendieck constant for a pair of Banach spaces, where the first space is finite-dimensional, as follows: 
\begin{definition} 
Let $(E,F)$ be a pair of Banach spaces with $E$ being finite-dimensional. Define the Grothendieck constant $K_G(E,F)$ to be the supremum  
\begin{eqnarray}\label{Definition-KG}
K_G(E,F):=\sup\left\{\|A\otimes \operatorname{id}_F\|_{E\check{\otimes}F\to E^*\hat{\otimes}F}:\|A\|_{E\to E^*}\leqslant 1\right\},
\end{eqnarray}
where $\operatorname{id}_F$ is the identity operator on $F$.
\end{definition} 
  Let $E$ be a finite-dimensional Banach space. In what follows, we employ the natural notion of positivity in $E\otimes E$, namely, $A\in E\otimes E$ is \textit{nonnegative} ($A\geqslant 0$) if it is in the convex hull of the set of symmetric tensors $e\otimes e$, $e \in E$. In other words, $A\geqslant 0$ if $A = B^*B$  for some $B \in E^*\otimes \ell_2^k$, $k\in\mathbb N$. If the supremum in \eqref{Definition-KG} defining $K_G(E,F)$ is restricted to non-negative matrices $A$, then we let the resulting constant $K_G^+(E, F)$ (say) be the {\it positive Grothendieck constant} corresponding to the pair $(E,F)$. Thus, 
\begin{equation}\label{Definition-KG+}
K_G^+(E,F) := \sup\left\{\|A\otimes \operatorname{id}_F\|_{E\check{\otimes}F\to E^*\hat{\otimes}F}:A \geqslant 0, \|A\|_{E\to E^*}\leqslant 1\right\}.\end{equation}
Evidently, for any pair of Banach spaces $E$ and $F$ we 
have $K_G^+(E,F)\leqslant K_G(E,F)$.

The positive Grothendieck constant $K_G^+:=\sup_{n\geqslant 1}K_G^{+}(\ell_\infty^n,\ell_2)$ has appeared in \cite[Theorem 4]{Rietz}. Although, it is also in \cite{Grothendieck} in a slightly different form. In the paper \cite{H}, the relationship of $K_G^+$ with the existence of orthogonally scattered dilations of Hilbert space-valued vector measures is discussed. More recently, in the PhD thesis of Bri\"{e}t \cite{Briet}, many variants of positive Grothendieck inequality and applications have been investigated. For any Hilbert space $\mathcal{H}$, the equality $\Pi_2(L_\infty,\mathcal H)= B(L_\infty,\mathcal H)$ is a manifestation of the finiteness of  $K_G^{+}$. Our motivation for defining the Grothendieck constant in greater generality is manifold, including but not limited to the following.

\begin{itemize}[leftmargin=\dimexpr 30pt]
    \item Recall that another reformulation of the Grothendieck inequality which is often referred to as the Grothendieck's theorem, is the equality $B(L_1,\ell_2) = \Pi_1(L_1,\ell_2)$. It is natural to ask which other Banach spaces possess such a property. 
A Banach space $E$ is said to be a G.T. space if $B(E,\ell_2) = \Pi_1(E,\ell_2)$, see \cite{Pisier-factorization}, via an equivalent norm.
The Grothendieck inequality \eqref{msc} establishes a non-trivial relationship between the three fundamental Banach spaces $\ell_\infty^n$, $\ell_1^m$, and $\ell_2$. A natural question arises: what happens if we replace $\ell_\infty^n$, $\ell_1^m$, and $\ell_2$ in the Grothendieck inequality \eqref{msc} with Banach spaces $E_n$, $E_m^*$, and $F$, respectively, where $(E_n)_{n \geqslant 1}$ is a sequence of finite-dimensional Banach spaces? Moreover, if $E_n$, $n\in\mathbb N$, and $F$  are taken to be finite-dimensional Banach spaces, then it is natural to expect that the ``quantitative'' information of the constant $K_G(E_n,F)$, now also depending on $n$, would lead to useful asymptotic behavior. 
    
    \item Recently, constants like $K_G(E,F)$ have appeared in quantum information theory and in particular XOR games in general probabilistic theories. For instance, similar constants have been studied in Proposition A.1 of \cite{XOR games and GT} to estimate the bias of a XOR game over a bipartite GPT under local strategies.

\item Like the Grothendieck inequality (and equivalently G.T.), a variant involving only non-negative definite matrices in \eqref{msc} (and equivalently the ``little G.T.") has been studied vigorously by many authors. We refer to \cite[Chapter 5]{Pisier-factorization} for more on this topic. In this article we prove that a finite-dimensional normed linear space $E$ has Property P, introduced earlier in \cite{BM}, if and only if $\sup_{m\geqslant 1}K^+_G(E,\ell^m_2) \leqslant 1$. The proof, in the real case, is provided in Proposition \ref{++++}. The proof for the complex case is similar and is omitted. This kind of interaction with Banach space geometry and operator theory is old and has been studied by several authors. We refer to the  papers \cite{vip}, \cite{BM} and \cite{AFHS95} and references therein for more information.
\end{itemize}

To facilitate the study of the Grothendieck constant $K_G(E,F)$ and its positive variant, we introduce the notion 
of a \textit{Grothendieck pair}. 

\begin{definition}[Grothendieck pair]\label{GTPDEF} Let $\mathcal{E}=(E_n)_{n\geqslant 1}$ be a sequence of finite-dimensional Banach spaces such that $\dim E_n=n.$ Let $F$ be a Banach space. Then $(\mathcal{E}, F)$ is called a {\textit Grothendieck pair} if there exists a constant $C>0$ such that $\|A\otimes \operatorname{id}_{F}\|_{E_m\check{\otimes}F\to E_n^*\hat{\otimes}F}\leqslant C \|A\|_{E_m\to E_n^*}$ for all  $m,n\in\mathbb{N}$. \end{definition}

 In Section 2, we gather definitions and preliminaries that we use throughout the paper. In Section 3, we prove several preparatory results. The main result of this paper is the following theorem and is proved in Section 4. 
\begin{theorem}\label{mainthm}
 Suppose that $(\mathcal E, F)$  is a Grothendieck pair. Then, both $F$ and $F^*$ are of non-trivial cotype.
\end{theorem} 

For a Banach space $X$, set 
 \[p(X):=\sup\{p:X\ \text{is of type }p\}\] and \[q(X):=\inf\{q:X\ \text{is of cotype } q\}.\]
 Suppose that  $(\mathcal E, F)$  is a Grothendieck pair and $F$ is an infinite-dimensional GL-space (see \cite[pp. 350]{DJT}). 
Then combining Theorem \ref{mainthm} with \cite[Theorem 17.13]{DJT}, it follows that
$p(F)>1$. Moreover, following the proof of Corollary \ref{cor4.6}, one concludes that $(\mathcal E,\ell_2)$ is then also a Grothendieck pair. 
It is natural to ask when $(\mathcal{E},F)$ is a Grothendieck pair. After communicating this question to Pisier, he made the following conjecture \cite{Pisier17}. 

\begin{conjecture}\label{Conpis}  Suppose that $(\mathcal{E},F)$ is a Grothendieck pair for a fixed but arbitrary Banach space $F$ with the bounded approximation property. Then either $\dim F<\infty$ or $(\mathcal{E},\ell_2)$ is also a Grothendieck pair.
\end{conjecture}
Clearly, by the previous discussion, Conjecture \ref{Conpis} is true  when $F$ is a GL-space. Surprisingly, Conjecture \ref{Conpis} is also related to an older conjecture by Pisier; see \cite[Final remarks (i)]{PisierDuality}.
\begin{conjecture}\label{Pisann}
If $X$ is an infinite-dimensional Banach space with bounded approximation property such that $q(X)<\infty$ and $q(X^*)<\infty,$ then $X$ is a $K$-convex space. 
\end{conjecture}
At the end of the paper \cite{Pisann}, under ``Added in proof'', the existence of a Banach space $X$ such that both $X$ and $X^*$ are of cotype $2$, although $X$ is not $K$-convex (so that $p(X) = 1$) was asserted. Furthermore, it was noted that such a space (which necessarily does not have the approximation property) contains uniformly complemented $\ell_p^n$'s for no $p$ such that $1 \leqslant p \leqslant \infty$.

In a private communication \cite{Pisier17}, Pisier had hinted that an affirmative answer to Conjecture \ref{Pisann} might establish Conjecture \ref{Conpis}. Corollary \ref{cor4.6} in this article verifies this assertion. 
This verification relies on Theorem \ref{mainthm}. The proof of Theorem \ref{mainthm} depends on a deep `$\ell_1/ \ell_2/ \ell_\infty$' trichotomy theorem recently proved in \cite{XOR games and GT}. 
In Proposition \ref{originalgt}, we give several equivalent conditions for a sequence of finite-dimensional Banach spaces $\mathcal E:=(E_n)_{n\geqslant 1}$ such that $(\mathcal E,\ell_2)$ is a Grothendieck pair. One of them says that $(E_n^*)_{n\geqslant 1}$ will have to satisfy G.T. uniformly. Proposition \ref{originalgt} was communicated to one of the authors by G. Pisier in an email message \cite{Pisier17} along with most of the proofs. 
 Along the way, we obtain asymptotic bounds for the constant $K_G(E,F)$ associated with several finite-dimensional Banach spaces, namely $\ell_p^n$ and $S_p^{n,\text{sa}}$, the set of all $n\times n$ self-adjoint matrices equipped with the Schatten $p$-norm. 
To avoid any confusion over the terminology, we point out that the notion of `GT pair'  introduced by Pisier in \cite{Pisier} is distinct from the Grothendieck pairs discussed in this paper.

  We begin our study of the positive variant $K_G^+(E,F)$ of the Grothendieck constant for various finite-dimensional Banach spaces $E$ and $F$ in Section \ref{S5}. The imposition of this additional condition makes the computation of $K_G^{+}(E,F)$ somewhat more difficult.  In this section, we also introduce the relevant notion of a \textit{G. T. flag} and find an interesting connection with Hilbert-Schmidt spaces, see Proposition \ref{prop:5.5}.  Also, we find that the behavior of $K_G^{+}(E,F)$ is quite different from that of $K_G(E,F)$. For example, $K_G(\ell_1^n,\ell_\infty^n)=o(\sqrt{n})$, whereas $K_G^{+}(\ell_1^n,\ell_\infty^n)$ is uniformly bounded by the Grothendieck constant (see Proposition \ref{positive-1-infinity}). Interestingly, the former estimate involves the real part of the Discrete Fourier Transform matrix. 
  
One of the key tools employed in our paper is the constant \( \rho(E, F) \), recently introduced in \cite{XOR games and GT} for finite-dimensional Banach spaces. It is defined as the operator norm of the identity map from the injective tensor product to the projective tensor product:
\[
\rho(E, F) := \left\| \mathrm{id}_E \otimes \mathrm{id}_F \right\|_{E \check{\otimes} F \to E \hat{\otimes} F}.
\]
Intuitively, this constant measures the maximal distortion between the injective and projective norms on \( E \otimes F \).
However, in this paper, we don't restrict the definition of $\rho$ to finite-dimensional Banach spaces. Indeed, it makes sense if one of $E$ and $F$ is finite-dimensional.    
Taking a finite-dimensional Banach space $E$ ($=F$), in the definition of $\rho$ and restricting to non-negative tensors of $E\otimes E$,
in the definition of $\rho$, we get a variant of the original $\rho$, and denote it by $\rho^+(E)$. For a real Banach space $E$,  Property Q 
introduced earlier in \cite{BM} is equivalent to requiring $\rho^+(E) \leqslant 1$. 
In this paper along with the Grothendieck constants, $K_G$ and $K_G^+$, we also study the constants $\rho$ and $\rho^+$. Among other things, Theorem \ref{rho+-1-1} in this section shows that $\rho^+(\ell_1^n) \geq c \sqrt{n}$ for some positive constant $c$. 

The comparison of the injective and projective tensor norms on finite-dimensional sequence spaces has also been studied in the theory of $\varepsilon - \pi$ continuous operators. Bonet, Defant, Peris and Ramanujan \cite[Proposition 3]{BDPR} and the subsequent remark determined the asymptotic order of $\rho(\ell_p^n,\ell_q^n)$. More recently, the corresponding problem for $\rho(\ell_p^n,\ell_q^m)$, allowing the two dimensions to vary independently is studied in the paper \cite{NPS}. They obtain sharp estimates when the two exponents lie on the same side of $2$, together with further estimates in the mixed range. Our proofs of the diagonal sequence-space estimates are included to exhibit their connection with the Grothendieck constants considered here and with the arguments for Schatten classes.

In Section \ref{adde}, we discuss applications of our results. These involve only the complex positive Grothedieck constant. In this Section, we recall in detail the relationship between Parrott homomorphisms and a class of linear maps between two finite dimensional Banach spaces. The contractivity versus the complete contractivity of these homomorphisms is measured by a constant $\alpha$ introduced by Paulsen in \cite{vip}. However, we discuss a smaller class of these that we call \textit{little Parrott homomorphisms} and show that in this smaller class, the measure of the contractivity versus complete contractivity is closely related to postive Grothendieck constant as defined in \eqref{Definition-KG+}, see Theorem \ref{thm5.17}. Moreover, we establish the exact asymptotic behavior of this constant for several naturally occurring Banach spaces, while the precise computation of \( \alpha \) for the corresponding spaces appears to remain widely open \cite{vip}. Here, we also provide a simple proof of the assertion: If $n\leqslant 3$, then $\sup_{m\geqslant 1}K_G^{+}(\ell_\infty^n,\ell_2^m)=1$ by  using bounds on the rank of the extreme points of correlation matrices  obtained  in \cite{LCB}. Moreover, following \cite{AFHS95}, we first show that $K_G^{+}(\ell_\infty^4,\ell_2^2)\geqslant 1.1658$ and then use it to  obtain a  quick proof of \cite[Theorem 2.1]{Davidchoi}.

\subsection*{Note added in revision} Subsequent to the earlier arXiv versions of this paper, the preprint \cite{chatgpt} presented a proof of the cotype–cotype conjecture under the ordinary approximation property. Specifically, \cite[Theorem 1.1]{chatgpt} states that a real Banach space with the approximation property is $K$-convex whenever both the space and its dual have finite cotype. This strengthens the assertion of Conjecture 1.5 by replacing the bounded approximation property with the approximation property. Combined with Theorem 1.3 and the argument of Corollary 4.8 below, this gives the conclusion of Conjecture 1.4 under the same weaker approximation hypothesis. This consequence is also recorded in \cite[Corollary 6.1]{chatgpt}.

\section{Preliminaries}\label{prel} 
Let $E$ and $F$ be Banach spaces. The norm of an operator $u:E\to F$ is denoted by $\|u\|_{E\to F}$ or $\|u\|$ whenever the meaning is clear from the context. We let $B(E,F)$ denote the linear space of all bounded linear maps from $E$ to $F$ equipped with the operator norm. The closed unit ball of $E$ is denoted by $(E)_1$. The Banach-Mazur distance $d(E,F)$ between two isomorphic Banach spaces $E$ and $F$ is defined as follows:
\begin{eqnarray*}
d(E,F):=\inf\big \{\,\|u\| \|u^{-1}\| \,\mid\, u:E\to F\,\, \mbox{\rm bounded invertible}\big \}.
\end{eqnarray*}
 If $d(E,F)\leqslant \lambda$ for some $\lambda>0$, then $E$ is said to be $\lambda$-isomorphic to $F$. 
The {\it factorization constant} of a Banach space $E$ through another Banach space $F$ is defined to be 
\begin{eqnarray*}
f(E,F):=\inf \big \{ \|u\| \|v\|\, \mid \,  u:E\to F,\, v:F\to E, \mbox{\rm and}\,v u= \mbox{\rm id}_E \big \},
\end{eqnarray*}
whenever it exists.
Evidently, $d(E,F)=f(E,F)$ whenever $E$ and $F$ are finite-dimensional Banach spaces with $\dim E=\dim F.$ Moreover, $f(E^*,F^*)\leqslant f(E,F)$ with equality if both $E$ and $F$ are finite-dimensional.

Let $E$ be a Banach space, and assume that $\lambda\geqslant 1$. We say that $E$ contains $\ell_p^n$'s $\lambda$-uniformly if there exists a sequence of subspaces $(E_n)_{n\geqslant 1}$ of $E$ such that $\sup_{n\geqslant 1} d(\ell_p^n, E_{n})\leqslant \lambda$.
Dvoretzky's theorem asserts that any infinite-dimensional Banach space $E$ contains $\ell_2^n$'s $(1+\epsilon)$-uniformly for all $\epsilon >0$.

\subsection{Norms on tensor product of Banach spaces:} In what follows we identify the algebraic tensor product $E\otimes F$ with a subspace of $B(E^*,F)$. Any tensor $u$ of the form $\sum_{j=1}^n e_j\otimes f_j$, with $e_j\in E$ and $f_j\in F,$ defines a linear map $u:E^* \to F$ by setting $u(e^*) := \sum_{j=1}^n e^*(e_j) f_j$, $e^*\in E^*$. 
The injective tensor norm $\|u\|_\vee$ is the operator norm $\|u\|_{E^* \to F}$. Moreover, the projective norm $\|u\|_\wedge$ is defined to be 
\[\|u\|_\wedge=\inf \Big \{ \sum_{j=1}^n \|e_j\|_E \|f_j\|_F \big| u = \sum_{j=1}^n e_j \otimes f_j \Big \},\] 
where the infimum is taken over all representations of $u.$
We let $E\check{\otimes}F$ and $E\hat{\otimes}F$ denote the completion of the linear space $E\otimes F$ equipped with the {\it injective} and {\it projective} tensor norms, respectively. If $E$ is finite-dimensional, we have the remarkable duality $(E\check{\otimes}F)^*\cong E^*\hat{\otimes}F^*$, and $(E\hat{\otimes}F)^*\cong E^*\check{\otimes}F^*$ via  the equality $\langle e\otimes f,e^*\otimes f^*\rangle=e^*(e)f^*(f).$
Note that the canonical operator $J:E^*\hat{\otimes}F\to B(E,F)$ defined as $J(e^*\otimes f)(e):=e^*(e)f,$ is an isomorphism, when $E$ and $F$ are finite-dimensional Banach spaces. When $E$ and $F$ are finite-dimensional, the nuclear norm of $u\in B(E,F)$ is defined to be $\|J^{-1}(u)\|_{E^*\hat{\otimes}F}$ and is denoted by $N(u)$. We denote $N(E,F)$ to be the linear space $B(E,F)$ equipped with the nuclear norm. 

We recall a very useful property of the norm, namely, let $u\in B(X,E)$, $w\in B(F,Y)$ and $v\in N(E,F)$. Then  
\[N(wvu) \leqslant \|u\|_{X\to E}\,\, N(v)\,\,\|w\|_{F\to Y}.\]

Our main reference for norms on the tensor product of Banach spaces and their properties is \cite{Raymond02}.
The following theorem due to S. Chevet (see \cite[Theorem 3.20]{LT}) is useful for estimating injective norm of random tensors.
\begin{theorem}[Chevet's theorem] \label{chevet}
Let $X$ and $Y$ be real finite-dimensional Banach spaces. Define the Gaussian random tensor $z=\sum_{i=1}^m\sum_{j=1}^ng_{ij}x_i\otimes y_j\in X\otimes Y,$ where $(g_{ij})$ are i.i.d $N(0,1)$ Gaussian random variable and $(x_i)_{i=1}^m\subseteq X$, $(y_j)_{j=1}^n\subseteq Y.$ Let $(g_i)_{i=1}^n$ be a sequence of i.i.d $N(0,1)$ Gaussian random variables. Then \[\mathbb{E}\|z\|_{X\check{\otimes}Y}\leqslant \|T\|_{\ell_2^m\to X}\mathbb{E}\Big\|\sum_{i=1}^mg_iy_i\Big\|_{Y}+\|S\|_{\ell_2^n\to Y}\mathbb{E}\Big\|\sum_{i=1}^ng_ix_i\Big\|_{X}\] where $T(\mathbf{e}_i):=x_i$ and $S(\mathbf{e}_j):=y_j$ for $1\leqslant i\leqslant m$, $1\leqslant j\leqslant n,$ and $(\mathbf{e}_i)_{i\geqslant 1}$ is the canonical basis of $\ell_2.$ 
\end{theorem}
\begin{definition}[$p$-summing operator]
Let $u:E\to F$ be a linear operator between two Banach spaces and $p\in [1,\infty)$. We say $u$ is $p$-summing if there exists a constant $C>0$ such that for any finite sequence $(x_i)$ in $E,$
$$\Big(\sum\|ux_i\|^p\Big)^{\frac{1}{p}}\leqslant C\sup\Big\{\big(\sum|x^*(x_i)|^p\big)^{\frac{1}{p}}:x^*\in (E^*)_1\Big\}.$$ 
Moreover, the best constant $C$ in the above inequality is denoted by $\pi_p(u)$ and is said to be the $p$-summing norm of $u.$ The set of all $p$-summing operators from $E$ to $F$ is denoted by $\Pi_p(E,F).$
\end{definition}
A linear operator $u:E\to F$ is said to factor through a Hilbert space if there is a Hilbert space $\mathcal H$  and linear operators $B:E\to\mathcal H$ and $A:\mathcal H\to F$ such that $u=AB$. We then define $\gamma_2$-norm of $u$ by $\gamma_2(u):=\inf\|A\|\|B\|,$ where the infimum runs over all possible factorizations of $u$ through a Hilbert space. The space of all linear operators from a Banach space $E$ to a Banach space $F,$ which factors through a Hilbert space, becomes a Banach space equipped with the $\gamma_2$-norm and is denoted by $\Gamma_2(E,F).$
\begin{definition} \label{GTunif}
 A family of Banach spaces $(E_n)_{n\geqslant 1}$ is said to satisfy GT uniformly if there exists $c>0$ such that for all $m,n\geqslant 1$ and any map $u:E_m\to\ell_2^n,$ we have $\pi_1(u)\leqslant c\|u\|_{E_m\to\ell_2^n}.$
\end{definition}
\subsection{Type, cotype and related notions:}
Let $(\epsilon_i)_{i=1}^n$ be a sequence of i.i.d. Bernoulli random variables taking values in $\{+1,-1\}$ with equal probabilities.
\begin{definition}
    A Banach space $E$ has Rademacher type $p$ (in short, type $p$) for some $1\leqslant p\leqslant 2$ if there is a constant $C>0$ such that for all $n\geqslant 1$ and $e_1,\dots,e_n\in E$ \[\Big(\mathbb E \Big\|\sum_{i=1}^n\epsilon_ie_i\Big\|_E^2\Big)^{\frac{1}{2}}\leqslant C\Big(\sum_{i=1}^n\|e_i\|^p\Big)^{\frac{1}{p}}.\] 
    The best constant $C$ in the above inequality is denoted by $T_p(E).$ 
    Analogously, a Banach space $E$ is said to have Rademacher cotype $q$ (in short, cotype $q$) if for some $2\leqslant q\leqslant\infty$ there is a constant $C>0$ such that for all $n\geqslant 1$ and $e_1,\dots,e_n\in E$ \[\Big( \mathbb E\Big\|\sum_{i=1}^n\epsilon_ie_i\Big\|_E^2\Big)^{\frac{1}{2}}\geqslant C^{-1}\Big(\sum_{i=1}^n\|e_i\|^q\Big)^{\frac{1}{q}}.\]The best constant $C$ in the inequality above is denoted by $C_q(E).$ 
\end{definition}
\begin{definition}
  A family of Banach spaces $(E_n)_{n\geq 1}$ is called uniformly of cotype $q$ if there exists a constant $M>0$ such that $C_q(E_n)\leqslant M$ for each $n\in\mathbb N.$  
\end{definition}
For each $n\in\mathbb N,$ let $G_n$ be the compact abelian group $\{+1,-1\}^n$ with the normalized Haar measure. Let $X$ be a Banach space. For any $f:G_n\to X,$ suppose $f=\sum_{\gamma\in \widehat{G}_n}\widehat{f}(\gamma)\gamma $ is its Fourier-Walsh expansion, where $\widehat{G}_n$ is the Pontryagin dual of $G_n$. Let $R_n:L_2(G_n;X)\to L_2(G_n;X)$ be the projection defined by $R_n(f):=\sum_{i=1}^n\widehat{f}(\epsilon_i)\epsilon_i.$ Define 
\[K(X):=\sup_{n\geqslant 1}\|R_n\|_{L_2(G_n;X)\to L_2(G_n;X)}.\]  
\begin{definition}
 We say a Banach space $X$ is {\it $K$-convex} if $K(X)$ is finite.   
\end{definition}
In Theorems \ref{Maurey-Pisier}, \ref{thm:2.6} and \ref{pieucli}, $F$  is an infinite-dimensional Banach space. The following theorem is due to Maurey and Pisier. 
\begin{theorem} \cite[Theorem 3.3(ii)]{Pis85}\label{Maurey-Pisier}
A Banach space $F$ has finite cotype if and only if $F$ does not contain $\ell_\infty^n$'s $\lambda$-uniformly for any $\lambda \geqslant  1.$
\end{theorem}
Theorem 5.4 of \cite{Pis85} says that a Banach space $F$ is $K$-convex if (and only if) it does not contain $\ell_1^n$'s uniformly. Combining this with Theorem 3.1(i) of the same paper \cite{Pis85}, we infer the following. 

\begin{theorem} \label{thm:2.6} 
A Banach space $F$ is $K$- convex if and only if $p(F)>1$. 
\end{theorem} 
A Banach space $F$ is said to be {\it locally $\pi$-Euclidean} if there exists a constant $C>0$ such that for each $\epsilon>0$ and each positive integer $n$, there is a positive integer $N(n,\epsilon)$ such that every subspace $E\subseteq F$ with $\text{dim}E \geqslant N$ contains an $n$-dimensional subspace $G\subseteq E$ such that $d(G,\ell_2^n)\leqslant 1+\epsilon$ and there is a projection from $F$ onto $G$ with norm less than $C.$ 
\begin{theorem}[Theorem 5.10, \cite{Pis85}]\label{pieucli}

A Banach space $F$ is locally $\pi$-Euclidean if and only if $F$ is $K$-convex.
\end{theorem}
We need the following remarkable `$\ell_1/ \ell_2/ \ell_\infty$' trichotomy theorem.
\begin{theorem}[Theorem 20, \cite{XOR games and GT}]\label{trichotomy} Suppose that $E$ is a Banach space of dimension $n$. Then, for every $1\leqslant A\leqslant\sqrt{n}$, there exist universal constants $c>0$ and $C> 0$, such that one of the following is true:
\begin{itemize}
    \item[(i)] $f(\ell_\infty^{c\sqrt{n}}, E)\leqslant CA\sqrt{\log n}$;
    \item[(ii)]$f(\ell_1^{c\sqrt{n}}, E)\leqslant CA\sqrt{\log n}$; 
    \item[(iii)] $f(\ell_2^{\frac{cA^2}{\log{n}}}, E)\leqslant C\ {\log n}.$
\end{itemize}
\end{theorem}
Let $M_n$ be the algebra of $n\times n$ complex matrices. 
For $1\leqslant p<\infty,$
define $\|A\|_{S_p^n}:=(\text{tr}(|A|^p))^{\frac{1}{p}}$. This makes $M_n$ a complex Banach space denoted by $S_p^n$. We denote $S_\infty^n$ to be $M_n$ equipped with the usual operator norm. Note that the space $S_p^{n,\text{sa}}$ of all self-adjoint elements of $S_p^n$ is a real Banach space equipped with the norm of $S_p^n.$ 
In what follows, $p,p'\in [1,\infty]$, are called conjugate to each other if $\frac{1}{p}+\frac{1}{p'}=1.$ The notation $p'$ shall always denote the conjugate of $p,$ unless stated otherwise.  
It is well known that $(S_p^n)^*$ is isometrically isomorphic to $S_{p'}^n.$
The duality relation is given by $\langle A, B\rangle=\operatorname{tr}(AB),$ where $A\in S_p^n$ and $B\in S_{p'}^n.$ 
An analogous result holds for $S_p^{n,\text{sa}}.$ 
We need the following non-commutative $L_p$-Grothendieck theorem.

We recall below, in the form of a theorem, Equation (0.2) of \cite{Xu06} that is stated for non-commutative $L_p(M)$ spaces over a von Neumann algebra $M$. However, since 
$L_p(M)=S_p$ with $M=B(\ell_2)$, we also have it for $S_p$. 

\begin{theorem}[pp. 527, \cite{Xu06}] \label{LpncG}
For any $1<p,q<\infty$ and any bounded bilinear form $B:S_{2p}\times S_{2q}\to \mathbb{C},$ there are positive unit functionals $\phi$ and $\psi$ on $S_p$ and $S_q$ respectively, 
such that 
\[|B(x,y)|\leqslant K_{p,q}\|B\|\Bigg(\phi\bigg(\frac{x^*x+xx^*}{2}\bigg)\Bigg)^{\frac{1}{2}}\Bigg(\psi\bigg(\frac{y^*y+yy^*}{2}\bigg)\Bigg)^{\frac{1}{2}},\ \forall\ x\in S_{2p},\ y\in S_{2q},\] 
for some positive constant $K_{p,q}$.
\end{theorem}
The paper \cite{Xu06} also has an extension of Theorem \ref{LpncG} for operator
spaces. 

\section{A number of preparatory lemmas}
In this section, we prove various results which will be useful for later sections. 
We start with the following useful lemma which gives an equivalent description of $K_G(E,F)$ for finite-dimensional Banach spaces $E$ and $F$, when $E_1=E_2=E$.

\begin{lemma}\label{Pisier-Nuclear}
Suppose that $E_i$, $i=1,2$ and $F$ are finite-dimensional Banach spaces. Then the following statements are equivalent.
\begin{itemize}
\item[1.] For all linear maps $A:E_1\to E_2^*$, there exists a constant $C>0$, independent of $A,$ such that
\begin{equation}\label{pis3}
\|A\otimes \operatorname{id}_{F}\|_{E_1\check{\otimes}F\to E_2^*\hat{\otimes}F}\leqslant C \|A\|_{\vee}.
\end{equation}
\item[2.] For all linear maps $A:E_1\to E_2^*$ and linear maps $B:E_1^*\to F,$ there exists a constant $C>0$, independent of $A$ and $B,$ such that 
\begin{eqnarray}\label{Nuclear}
N(BA^*)\leqslant C\|A\|_{\vee}\|B\|_{\vee}.
\end{eqnarray}
\end{itemize}
{\textit{Moreover, the best constants in both \eqref{pis3} and \eqref{Nuclear} are equal.}}
\end{lemma}
\begin{proof}
Fix bases as $(e_i^1)_{i=1}^{\text{dim}E_1}$ of $E_1$, $(e_i^2)_{i=1}^{\text{dim}E_2}$ of $E_2$ and $(f_j)_{j=1}^{\text{dim}F}$ of $F.$ Let the dual bases be $({e_i^1}^*)_{i=1}^{\text{dim}E_1}$ and $({e_i^2}^*)_{i=1}^{\text{dim}E_2}$ for the dual space $E_1^*$ and $E_2^*$ respectively.
Let us consider \[\sum_{i=1}^{\text{dim}E_1}\sum_{j=1}^{\text{dim}F}b_{ij}e_i^1\otimes f_j \] to be an arbitrary element in $E_1\otimes F$  and let $A\in E_1^*\check{\otimes} E_2^*.$ Suppose $A$ is represented by $Ae_i^1=\sum_{k=1}^{\text{dim}E_2}a_{ki}{e_{k}^2}^*.$ Note that 
\[(A\otimes \text{id}_{F})\Big(\sum_{i=1}^{\text{dim}E_1}\sum_{j=1}^{\text{dim}F}b_{ij}e_i^1\otimes f_j \Big)=\sum_{i=1}^{\text{dim}E_1}\sum_{j=1}^{\text{dim}F}b_{ij}A(e_i^1)\otimes f_j=\sum_{i=1}^{\text{dim}E_1}\sum_{j=1}^{\text{dim}F} b_{ij}\sum_{k=1}^{\text{dim}E_2}a_{ki}{e_
{k}^2}^*\otimes f_j.\]
Hence
\begin{eqnarray*}
\Big\|(A\otimes \text{id}_{F})\Big(\sum_{i=1}^{\text{dim}E_1}\sum_{j=1}^{\text{dim}F}b_{ij}e_i^1\otimes f_j\Big)\Big\|_{E_2^*\hat{\otimes}F}= N(BA^*),
\end{eqnarray*}
where $A^*$ denotes the dual of $A$. Thus \eqref{pis3} is equivalent to 
\begin{eqnarray*}
N(BA^*)\leqslant C \|A^*\|_\vee \|B\|_\vee= C \|A\|_\vee \|B\|_\vee. 
\end{eqnarray*}
This proves the equivalence of \eqref{pis3} and \eqref{Nuclear}.
\end{proof}

\begin{remark}\label{gttheo}
Let $E$ and $F$ be finite-dimensional Banach spaces. Since $\pi_1$ is a cross norm and projective norm is the largest cross norm, in view of Lemma \ref{Pisier-Nuclear}, it follows that  for any pair of maps $A:E\to E^*$ and $B:E^*\to F$ we have 
\begin{equation}\label{pisier2}
 \pi_1(BA)\leqslant K_G(E,F)\|A\| \|B\|.
\end{equation}
\end{remark}
In what follows, we shall also need the following lemma. The proof is essentially same as the proof in \cite[Proposition 2.1(b)]{AFHS95}. We provide a proof for the sake of completeness.       \begin{lemma}\label{Norm of inclusion of pi-2(E,H)}
        For any finite-dimensional Banach space $E$ and any Hilbert space $\mathcal H,$ 
    \[\|\text{id}\|_{\Pi_2(E^*,\mathcal H)\to B(E^*,\mathcal H)}=\|\text{id}\|_{\Pi_2(E,\mathcal H)\to B(E,\mathcal H)}.\]
    \end{lemma}
    \begin{proof}
        It is enough to show that $\|\text{id}\|_{\Pi_2(E,\mathcal H)\to B(E,\mathcal H)} \leqslant \|\text{id}\|_{\Pi_2(E^*,\mathcal H)\to B(E^*,\mathcal H)}=K$ (say). To get this, let $T:E\to\mathcal H$ be a bounded linear operator.  
 Recall from \cite[Chapter 1]{Pisier-factorization} that $\pi_2(T)$ can be also expressed as 
\begin{equation}\label{formufrpi}
    \pi_2(T)=\sup\big\{\pi_2(Tu):\|u\|_{\ell_2^k\to E}\leqslant 1, k\in\mathbb{N}\big\}.
    \end{equation}
Also, note that \cite[Proposition 1.9]{Pisier-factorization} shows that $\pi_2(Tu)=\|Tu\|_{\text{HS}}$, where $\|
\cdot\|_{\text{HS}}$ is the usual Hilbert-Schmidt norm. Hence we get that 
\begin{equation}\label{formpi2}
\pi_2(Tu)=\pi_2(u^*T^*)\leqslant \|T^*\|\pi_2(u^*)\leqslant K\|T^*\|\|u^*\|=K\|T\|\|u\|.
\end{equation} 
This proves that $\pi_2(T)\leqslant K \|T\|$ and completes the proof of the lemma. 
    \end{proof}
\begin{lemma} \label{basic-properties-KG}
Let $E$ and $F$ be Banach spaces. Assume that $E$ is finite-dimensional. Then the following are true:
	\begin{itemize}
		\item[(i)] $K_G(E,F^*)= K_G(E,F).$
		\item[(ii)] $K_G(E,F)\leqslant \min\{\rho(E,F),\, \rho(E,F^*)\}.$
		\item[(iii)]  $K_G(\ell_2^n,F)=\rho(\ell_2^n,F).$
		\end{itemize}
\end{lemma}
\begin{proof} Let $A:E\to E^*$ be a linear map.
	\begin{itemize}
		\item[(i)] From the duality of projective and injective norm, we know that 
		\begin{eqnarray*}
		\|A\otimes \text{id}_F\|_{E\check{\otimes}F\to E^*\hat{\otimes}F} = \|A^*\otimes \text{id}_{F^*}\|_{E\check{\otimes}F^*\to E^*\hat{\otimes}F^*}
		\end{eqnarray*}
Since $\|A\|=\|A^*\|,$ it follows that $K_G(E,F)=K_G(E,F^*).$
		
\item[(ii)] Note that $A\otimes \text{id}_F=(A\otimes\text{id}_F)(\text{id}_{E}\otimes\text{id}_F).$ Therefore, we have that 
\begin{eqnarray*}
\|(A\otimes \text{id}_F)\|_{E\check{\otimes}F\to E^*\hat{\otimes}F} &=& \|(A\otimes \text{id}_F)\circ (\text{id}_E\otimes\text{id}_F)\|_{E\check{\otimes}F\to E^*\hat{\otimes}F}\\
&\leqslant & \|A\otimes\text{id}_F\|_{E\hat{\otimes}F\to E^*\hat{\otimes}F}\|\text{id}_{E}\otimes\text{id}_F\|_{E\check{\otimes}F \to E\hat{\otimes}F}.
\end{eqnarray*}
Since $ \|A\otimes\text{id}_F\|_{E\hat{\otimes}F\to E^*\hat{\otimes}F}=\|A\|,$ we have 
\[\|A\otimes\text{id}_F\|_{E\hat{\otimes}F\to E^*\hat{\otimes}F}\leqslant \rho(E,F) \|A\|.\] Therefore, $K_G(E,F)\leqslant \rho(E,F).$ The result follows from part (i).
\item[(iii)] Let us choose $A=\text{id}_{\ell_2^n}.$ From the definition of $K_G(\ell_2^n,F)$,  we get that 
\[K_G(\ell_2^n,F)\geqslant \|\text{id}_{\ell_2^n}\otimes\text{id}_F\|_{\ell_2^n\check{\otimes}F\to \ell_2^n\hat{\otimes}F}=\rho(\ell_2^n,F).\]
\end{itemize}
The  inequality on the other side  follows from (ii). 
 This completes the proof of the lemma.
\end{proof}

\begin{lemma}\label{KG and Subspaces}
Let $E$ and $F$ be Banach spaces. Let $X$ and $Y$ be another pair of Banach spaces. Assume that $E$ and $X$ are finite-dimensional 
and that $f(Y,F)$ and $f(Y,F^*)$ exist. 
Then the following are true:
\begin{itemize}
 \item[(i)] 
$K_G(E,Y)\leqslant \min\{f(Y,F), f(Y,F^*)\}K_G(E,F).$ Moreover,
\[K_G^{+}(E,Y)\leqslant \min\{f(Y,F), f(Y,F^*)\}K_G^{+}(E,F).\]
 \item[(ii)] 
 if $\dim X\leqslant\dim E$ then 
 $K_G(X,F)\leqslant {f}(X,E)^2 K_G(E,F).$ Moreover,  \[K_G^{+}(X,F)\leqslant {f}(X,E)^2 K_G^{+}(E,F).\]
\end{itemize}
\end{lemma}
\begin{proof}
 (i)~Suppose $A:E\to E^*$ is a contraction. Let $(u,v)$ be a pair of operators such that $vu=\text{id}_{Y}$ with $u:Y\to F$ and $v:F\to Y.$ Note that \[A\otimes \text{id}_{Y}=A\otimes vu=(\text{id}_{E^*}\otimes v)\circ(A\otimes\text{id}_F)\circ (\text{id}_E\otimes u).\] Therefore, we have that
 \begin{eqnarray*}\|A\otimes\text{id}_{Y}\|_{E\check{\otimes}Y\to E^*\hat{\otimes}Y}
&\leqslant &\|\text{id}_{E^*}\otimes v\|_{E^*\hat{\otimes}F\to E^*\hat{\otimes}Y}
\|A\otimes\text{id}_F\|_{E\check{\otimes}F\to E^*\hat{\otimes}F}\|\text{id}_{E}\otimes u\|_{E\check{\otimes}Y\to E\check{\otimes}F}\\
&\leqslant& \|v\|\|u\|K_G(E,F).
\end{eqnarray*}
 Now taking infimum over all pair $(u,v)$ such that $vu=\text{id}_Y$ in the above computation, we get  $K_G(E,Y)\leqslant f(Y,F)K_G(E,F).$ By Lemma \ref{basic-properties-KG}, we have $K_G(E,F^*)=K_G(E,F),$ and hence
 \begin{eqnarray*}
 K_G(E,Y)\leqslant \min\{f(Y,F),f(Y,F^*)\} K_G(E,F)
 \end{eqnarray*}
 The proof of the `positive' case is the same as the above.

(ii) Suppose $A:X\to X^*$ is a contraction. 
Choose pair of operators $(u,v)$ such that $vu=\text{id}_{X}.$ Note that we have $A\otimes\text{id}_F=u^*v^*Avu\otimes\text{id}_F.$ Therefore, we have that
\[A\otimes \text{id}_F=(u^*\otimes\text{id}_F)\circ (v^*\otimes\text{id}_F)\circ (A\otimes\text{id}_F)\circ (v\otimes\text{id}_F)\circ (u\otimes\text{id}_F).\] 
Therefore, we obtain that 
\begin{equation*}
\begin{split}
&\|A\otimes\text{id}_F\|_{X\check{\otimes}F\to X^*\hat{\otimes}F}\\
&\phantom{Ohm}=\|(u^*\otimes\text{id}_F)\circ (v^*Av\otimes\text{id}_F)\circ (u\otimes\text{id}_F)\|_{X\check{\otimes}F\to X^*\hat{\otimes}F}\\
&\phantom{Ohm}\leqslant\|u\otimes \text{id}_F\|_{X\check{\otimes}F\to E\check{\otimes}F}\|v^*Av\otimes \text{id}_F\|_{E\check{\otimes}F\to E^*\hat{\otimes}F}\|u^*\otimes\text{id}_F\|_{E^*\hat{\otimes}F\to X^*\hat{\otimes}F}\\
&\phantom{Ohm}\leqslant\|u\|_{X\to E}^2K_G(E,F)\|v^*Av\|_{E\to E^*}\\
&\phantom{Ohm}\leqslant\|u\|_{E_1\to E}^2\|v\|_{E\to E_1}^2.
\end{split}
\end{equation*}
Taking the infimum over all admissible $(u,v)$, we obtain the required result.
 The proof of the `positive' case also follows.
This completes the proof of lemma.
\end{proof}

\section{Asymptotic behavior of $K_G(E,F)$}
It follows from Equation \eqref{msc} that $\left((\ell_\infty^n)_{n\geqslant 1},\ell_2\right)$ is an example of Grothendieck pair. Moreover, the question of characterizing $(E_n)_{n\geqslant 1}$ such that $\left((E_n)_{n\geqslant 1},\ell_2\right)$ is a Grothendieck pair is very similar to the notion of GT spaces. This has been studied in the literature (see \cite{Pisier-factorization}).  

\begin{proposition}\label{originalgt} Let $\mathcal{E}=(E_n)_{n\geqslant 1}$ be a sequence of finite-dimensional Banach spaces with $\text{dim}\, E_n=n,$ $n\in\mathbb{N}.$
The following statements are equivalent. 
\begin{enumerate}
\item[(1)] $(\mathcal{E}, \ell_2)$ is a {Grothendieck pair}.
\item[(2)] There are positive constants $K_1$ and $K_2$ such that 
\begin{enumerate}
\item[(i)] for any Hilbert space $\mathcal H$ and for $B:E_n^*\to \mathcal{H}$,  
$\pi_1(B)\leqslant K_1\|B\|$, for all $n \geqslant 1,$ 
\item[(ii)] for any $A:E_m\to E_n^*$, 
$\gamma_2(A)\leqslant K_2\|A\|$, for all $m,n \geqslant 1$.
\end{enumerate}
\item[(3)] There are positive constants $K_1$ and $K_2$ such that 
\begin{enumerate}
\item[(i)] for any Hilbert space $\mathcal H$ and for $B:E_n^*\to \mathcal{H}$,  
$\pi_2(B)\leqslant K_1\|B\|$, for all $n \geqslant 1$, 
\item[(ii)] for any $A:E_m\to E_n^*$, 
$\gamma_2(A)\leqslant K_2\|A\|$, for all $m,n \geqslant 1$.
\end{enumerate}
\item[(4)] There is a constant $K>0$ such that  $\gamma_2^*(A)\leqslant K\|A\|$ for all $A:E_m\to E_n^*$ and for all $m,n\geqslant 1.$
\end{enumerate}
\end{proposition}
\begin{proof}
\textbf{$(1) \implies (2)$}: In what follows, we let $C>0$ be as in Definition \ref{GTPDEF}.

 Let $n\in\mathbb N$ and $B:E_n^*\to\ell_2^k$ be a linear operator for some $k\in\mathbb{N}$. Fix $\epsilon>0.$ Then by Dvoretzky's theorem for some large $m\in\mathbb{N}$ there exists $(1+\epsilon)$-isometry $j:\ell_2^k\to E_m.$ Let $A:E_m\to E_n^*$ be a linear operator. 
From (1), note that 
\begin{eqnarray}\label{Dvoretzky}
|\operatorname{tr}(jBA)|\leqslant N(jBA) \leqslant \|j\| N(BA) \leqslant C\|j\| \|A\| \|B\|,
\end{eqnarray}
where the last inequality follows from Lemma \ref{Pisier-Nuclear}.
Taking supremum over $A\in (E_m^*\check{\otimes}E_n^*)_1$  in \eqref{Dvoretzky}, we get 
\begin{eqnarray*}
N(jB)\leqslant C(1+\epsilon) \|B\|.
\end{eqnarray*}
Since $\pi_1(B)\leqslant\pi_1(jB)\leqslant N(jB),$ we get the first part of (2). 

Now for the second part, let $A:E_m\to E_n^*$ be a map and let $v:E_n^*\to E_m$ be a map. Consider a factorization of $v$ as $v=\alpha B,$ where $B:E_n^*\to \ell_2^k$ and $\alpha:\ell_2^k\to E_m$ for some $k\in\mathbb{N}.$ Note that 
\begin{eqnarray*}
N(vA)=N(\alpha B A)\leqslant \|\alpha\| N(BA) \leqslant C \|\alpha\| \|B\| \|A\|.
\end{eqnarray*}
Taking infimum over $B$ and $\alpha$ such that $v=\alpha B$ in the above inequality, we get
\begin{eqnarray}\label{gamma-2(v)}
|tr(vA)|\leqslant N(vA) \leqslant C \gamma_2(v) \|A\|.
\end{eqnarray}
We take supremum over $v$ in \eqref{gamma-2(v)} such that $\gamma_2(v)\leqslant 1,$ and obtain  
\begin{eqnarray}\label{gamma-2-*}
\gamma_2^*(A)\leqslant C \|A\|,
\end{eqnarray}
where $\gamma_2^*$ is the dual norm corresponding to the norm $\gamma_2.$ From \cite[Proposition 2.11]{Pisier-factorization}, we know that
\begin{eqnarray}\label{gamma*}
\gamma_2^*(A)=\inf\{\pi_2(\phi)\pi_2(\psi^*):\phi:E_m\to\ell_2,\ \psi:\ell_2\to E_n^*, A=\psi\phi\}.
\end{eqnarray}
By definition, $\gamma_2(A)=\inf\{\|\phi\| \|\psi\|:\phi:E_m\to\ell_2,\ \psi:\ell_2\to E_n^*,A=\psi\phi\}.$ Hence from \eqref{gamma-2-*}, we have
\begin{eqnarray*}
\gamma_2(A)\leqslant\gamma_2^*(A)\leqslant C\|A\|.
\end{eqnarray*}
This completes the proof of $(1) \implies (2)$ and also $(1)\implies (4).$

\textbf{$(2) \implies (3)$}: This follows from the fact that $\pi_2(\cdot)\leqslant \pi_1(\cdot).$ 

\textbf{$(3) \implies (1)$}:
Let $A:E_m\to E_n^*$ be any linear map. Choose linear maps $S:E_m\to\mathcal{H}$ and $T:E_n\to\mathcal{H}$ such that $A=T^*S$. 
From the hypothesis in (3), it is clear that $\|id\|_{\prod_2(E_n^*,\mathcal H)\to B(E_n^*,\mathcal H)}\leqslant K_1$ for each $n\in\mathbb N.$ From Lemma \ref{Norm of inclusion of pi-2(E,H)}, it follows that $\|id\|_{\prod_2(E_n,\mathcal H)\to B(E_n,\mathcal H)}\leqslant K_1$ for each $n\in\mathbb N.$ Hence $\pi_2(T)\leqslant K_1\|T\|$ and $\pi_2(S)\leqslant K_1 \|S\|.$
These put together with \eqref{gamma*}, it follows that
\[\gamma_2^*(A)\leqslant \pi_2(S)\pi_2(T)\leqslant K_1^2\|S\| \|T\|.\] 
Taking infimum over $S$ and $T,$ we obtain 
\begin{eqnarray}\label{gamma2* in 3 to 1}
\gamma_2^*(A)\leqslant K_1^2 \gamma_2(A)\leqslant K_1^2K_2\|A\|.
\end{eqnarray}
Take $B:E_n^*\to\mathcal{H}$ such that $\|B\| \leqslant 1.$ Then observe that from trace duality and \eqref{gamma2* in 3 to 1}, we have
\begin{eqnarray*}
N(BA) &=& \sup\{|\operatorname{tr}(DBA)|:\|D\|_{{\mathcal{H}\to E_m}} \leqslant 1\}\\ 
& \leqslant & \sup_{\|D\|\leqslant 1} \gamma_2^*(A) \gamma_2(DB)\\
& \leqslant & \sup_{\|D\|\leqslant 1} \|B\| \|D\| \gamma_2^*(A)\\
& \leqslant & K_1^2 K_2 \|A\|. 
\end{eqnarray*}
This completes the proof of (3)$\implies$(1).

\textbf{$(1) \implies (4)$}: While proving $(1)$ implies $(2)$, we have noted that the inequality in \eqref{gamma-2-*} follows from the assumptions in the statement $(1)$. This proves the assertion in $(4)$. 

\textbf{$(4) \implies (3)$}: Since  $\gamma_2(A)\leqslant \gamma_2^*(A)$ for any $A:E_m\to E_n^*$, part (ii) of (3) follows. Therefore, we only need to show that there is a constant $K_1 > 0$ such that for all Hilbert spaces $\mathcal{H}$ and $B:E_n^*\to\mathcal{H},$ $\pi_2(B)\leqslant K_1\|B\|$, $n\in\mathbb{N}$. Moreover, 
in the view of Lemma \ref{Norm of inclusion of pi-2(E,H)},
it is enough to show that for all $A:E_n\to\mathcal{H}$, $\pi_2(A)\leqslant K_1\|A\|.$ To this end, we claim that that $\gamma_2^*(A^*A)=\pi_2(A)^2.$ 
It is clear from \eqref{gamma*} that $\gamma_2^*(A^*A)\leqslant \pi_2(A)^2.$ To show the reverse inequality, note that (see \cite[Chapter 2]{Pisier-factorization} and  \cite[Chapter 7]{DJT})
\[\gamma_2^*(A^*A)=\sup\{|\operatorname{tr}(vA^*A)|:v:E_n^*\to E_n,\gamma_2(v)\leqslant 1\}.\] Hence for any $u:\ell_2^k\to E_n,$ with $\|u\|\leqslant 1,$ we have $\gamma_2^*(A^*A)\geqslant |\operatorname{tr}(uu^*A^*A)|=|\operatorname{tr}(Au(Au)^*)|=\pi_2(Au)^2,$ where in the last equality we have used that for an operator between Hilbert spaces, $2$-summing norm and Hilbert-Schmidt norms coincide. Hence we get \[\gamma_2^*(A^*A)\geqslant\sup\{\pi_2(Au)^2:\|u\|_{\ell_2^k\to E_n}\leqslant 1\}=\pi_2(A)^2,\] where in the last equality we have used \eqref{formufrpi}. This completes the proof of the claim. Now using hypothesis in (4), we get 
$$\pi_2(A)=\sqrt{\gamma_2^*(A^*A)}\leqslant \sqrt{K} \sqrt{\|A^*A\|}=\sqrt{K}\|A\|.$$
Choosing $K_1=\sqrt{K}$ completes the proof of $(4) \implies (3)$ and in turn completes the proof of the proposition.
\end{proof}
The corollary below gives a criterion for determining when $(E_n, \ell_2)$ is a Grothendieck pair in terms of the notion of uniform GT introduced earlier in Definition \ref{GTunif}.
\begin{corollary}
    Let $(E_n)_{n\geqslant 1}$ be a sequence of finite-dimensional Banach spaces with $\dim E_n=n.$ Then $(E_n,\ell_2)$ is a Grothendieck pair if 
    \begin{itemize}
        \item[(i)] the family $(E_n^*)_{\geqslant 1}$ is uniformly of cotype $2$ and
        \item[(ii)] the family $(E_n^*)_{n\geqslant 1}$ satisfies GT uniformly.
    \end{itemize}
\end{corollary}
\begin{proof}
The proof follows easily from part (2) of Proposition \ref{originalgt} and \cite[Theorem 4.1]{Pisier-factorization}.
\end{proof}
We now use the above corollary to produce some examples of Grothendieck pair.
\begin{ex}\label{xample}
\begin{itemize*}
    \item [(i)] The pair $(\ell_\infty^n,\ell_2)$ is a Grothendieck pair.
\newline    
    \item[(ii)] Let $E_n:=\text{span}\{1,z,z^2,\dots,z^n\}$ be the space of polynomials of degree at most $n$ in $H_\infty(\mathbb{T})$, where $H_\infty(\mathbb{T})$ is the Hardy space. Then $((E_n)_{n\geqslant  1},\ell_2)$ is a Grothendieck pair.
\end{itemize*}  
\end{ex}
\begin{proof}
    Since the pair $(\ell_\infty^n,\ell_2)$ is a Grothendieck pair by definition, we only prove (ii). 
 Note that $H_\infty(\mathbb{T})^*$ is of  cotype $2$ as $H_\infty(\mathbb{T})^*=(L_1(\mathbb{T})/H_1(\mathbb{T}))^{**}$ and the latter space is of cotype $2$ \cite{Pisier-factorization}. Note that the $\ell_\infty$-direct sum $(\Sigma_{n=0}^\infty E_n)_\infty$ is isomorphic to $H_\infty(\mathbb{T})$ by Bourgain and Pełczyński \cite{Bout3}. Now to complete the proof, note that Bourgain \cite{Bourgain} proved that $L_1(\mathbb{T})/H_1(\mathbb{T})$ is a GT space and hence the double dual $H_\infty(\mathbb{T})$ is also a GT space \cite[Proposition 6.2]{Pisier-factorization}.
\end{proof}

\begin{lemma} \label{lem:4.2}
Suppose that $E$  is a finite-dimensional Banach space. Then $K_G(E,E^*)$ $=$ $K_G(E,E)$  $=$  $K_G(E^*,E)$ $=$ $\rho(E,E)$.
\end{lemma}
\begin{proof}
Let $A:E\to E^*$ be any linear map. Taking $B=\text{id}_{E^*}$ in \eqref{Nuclear} of Lemma \ref{Pisier-Nuclear}, it follows that $N(A)\leqslant K_G(E,E^*)\|A\|_{E\to E^*}.$
This shows that $\|A\|_{E^*\hat{\otimes}E^*}\leqslant K_G(E,E^*)\|A\|_{E^*\check{\otimes}E^*}$ for all linear maps $A:E\to E^*,$ and consequently we get $\rho(E^*,E^*)\leqslant K_G(E,E^*).$ 
On the other hand, by Lemma \ref{basic-properties-KG}, we know that $K_G(E,E^*)=K_G(E,E)\leqslant \rho(E,E)$. Moreover, $\rho(E^*,E^*)=\rho(E,E)$ (see \cite[Proposition 12]{XOR games and GT}). Therefore, we conclude that $K_G(E,E)=\rho(E,E).$ 
\end{proof}
The corollary below follows easily from Lemma \ref{lem:4.2}. 
\begin{corollary}\label{impcor}
Let $n\geqslant 1,$ $1\leqslant p\leqslant \infty.$ 
Then,
\begin{itemize}
\item[(i)]$
K_G(\ell_p^n,\ell_p^n)=K_G(\ell_p^n,\ell_{p^\prime}^n)=\rho(\ell_p^n,\ell_p^n).$
\item[(ii)]$
K_G(S_p^{n,\text{sa}},S_p^{n,\text{sa}})=K_G(S_p^{n,\text{sa}},S_{p^\prime}^{n,\text{sa}})=\rho(S_p^{n,\text{sa}},S_p^{n,\text{sa}}).$
\end{itemize}
\end{corollary}
We gather some consequences of Corollary \ref{impcor} in the following lemma. 
The assertion of part (i), for instance, follows from Corollary \ref{impcor} and \cite[Page 697]{XOR games and GT}. Also, (ii) follows from Corollary \ref{impcor} and \cite[Proposition 12]{XOR games and GT}. Part (iii) follows by similar reasoning.
\begin{lemma}\label{prefin}
Suppose $n\geqslant 1.$ Then,
\begin{itemize}
\item[(i)] there exist positive constants $c_1$ and $c_2$ such that $c_1\sqrt{n}\leqslant K_G(\ell_\infty^n,\ell_\infty^n)\leqslant c_2\sqrt{n}$;
\item[(ii)] $K_G(\ell_2^n,\ell_2^n)=\rho(\ell_2^n,\ell_2^n)=n$, and 
\item[(iii)] there exist positive constants $c_1$ and $c_2$ such that $c_1\sqrt{n}\leqslant K_G(\ell_1^n,\ell_\infty^n)=\rho(\ell_1^n,\ell_1^n)\leqslant c_2\sqrt{n}$.
\end{itemize}
\end{lemma}
	In the following proof, if there exists a universal constant C such that $A_n \geqslant C B_n$, $n\in \mathbb{N}$,  then we write $A_n \gtrsim  B_n$.
\begin{proof}[{Proof of Theorem \ref{mainthm}}:] Suppose that $(\mathcal E, F)$  is a Grothendieck pair. To prove the theorem, we may assume without loss of generality that $\text{dim}\,F$ is not finite.
	Note that by Theorem \ref{Maurey-Pisier}, it is enough to show that both $F$ and $F^*$ do not contain $\ell_\infty^n$'s uniformly. From Lemma \ref{basic-properties-KG}, it is enough to show that $F$ does not contain $\ell_\infty^n$'s uniformly. We use the method of contradiction to prove this. Suppose $F$ contains $\ell_\infty^n$'s uniformly. Then  $\sup_{n\geqslant 1}f(\ell_\infty^n, F)<\infty.$ Therefore, by Theorem \ref{trichotomy}, namely, the `$\ell_1/\ell_2/\ell_\infty$'-  trichotomy, Lemma \ref{basic-properties-KG}, Lemma \ref{KG and Subspaces}, Lemma \ref{prefin} and \cite[Equation (59)]{XOR games and GT},  we have the following 
	\begin{itemize}
\item[(i)]$K_G(E_{_n},F)\gtrsim f(\ell_\infty^{c\sqrt{n}},E)^{-2}K_G(\ell_\infty^{c\sqrt{n}},\ell_\infty^{c\sqrt{n}}){\gtrsim}\frac{n^{\frac{1}{4}}}{A^2{\log n}}$ 
	\item[(ii)] $K_G(E_{_n},F)\gtrsim f(\ell_1^{c\sqrt{n}},E)^{-2}K_G(\ell_1^{c\sqrt{n}},\ell_\infty^{c\sqrt{n}})\gtrsim\frac{n^{\frac{1}{4}}}{A^2{\log n}}$
	\item[(iii)] $K_G(E_n,F)\gtrsim f(\ell_2^{cA^2/\log{n}},E)^{-2}K_G(\ell_2^{cA^2/\log{n}},\ell_\infty^{cA^2/\log{n}})\gtrsim\frac{A^2}{\log^3n}$
	\end{itemize}
Now taking $A=n^{1/16}$, we see that the rightmost quantities in the above three inequalities tend to infinity as $n$ approaches to infinity. This contradicts the assumption that $(\mathcal E, F)$  is a Grothendieck pair.
\end{proof}	
As explained in the introduction, we have an immediate corollary of Theorem \ref{mainthm}. 
\begin{corollary}
Suppose that  $(\mathcal E, F)$  is a Grothendieck pair and $F$ is an infinite-dimensional GL-space.
Then $(\mathcal E, \ell_2)$ is a Grothendieck pair.
\end{corollary}
\begin{proof}
    Note that by Theorem \ref{mainthm} and \cite[Theorem 17.13]{DJT}, we have $p(F)>1.$ Now the implication follows from Theorem 2.6, Theorem 2.7 and Lemma 3.4.
\end{proof}
We note that \cite[Theorem 17.13]{DJT} provides many examples of Banach spaces, where the assertion of the Conjecture \ref{Pisann} is evident.  
\begin{corollary}\label{cor4.6}  An affirmative answer to Conjecture \ref{Pisann} verifies the validity of Conjecture \ref{Conpis}.
\end{corollary}
\begin{proof}
Assume that $\text{dim}\,F=\infty.$ Note that if Conjecture \ref{Pisann} is assumed to have an affirmative answer, then by Theorem \ref{mainthm} and Theorem \ref{pieucli}, $F$ contains uniformly complemented $\ell_2^n$'s, since $F$ is $K$-convex. The proof is completed by applying (ii) of Lemma \ref{KG and Subspaces}.
\end{proof}
The corollary stated below follows from Lemma \ref{basic-properties-KG} and \cite[Proposition 12]{XOR games and GT}.
\begin{corollary}
Let $(E_n)$ be a sequence of finite-dimensional Banach spaces and $F$ be another Banach space. Suppose that $\sup_{n\geqslant 1}K_G(E_n,F)<\infty.$ Then  $\sup_{n\geqslant 1}\text{dim}\,E_n<\infty$ or $\text{dim}\,F<\infty$ or both $F$ and $F^*$ are of nontrivial cotype.
\end{corollary}
Now we investigate $K_G(E,F)$ for various Banach spaces $E$ and $F.$ We begin with the following proposition, which is well known when $F=\ell_2.$
\begin{proposition}
If $n\geqslant 2$ and $\text{dim}\, F\geqslant 2$, then $K_G(\ell_\infty^n,F)\geqslant \sqrt{2}.$
\end{proposition}
\begin{proof}
Note that by Lemma \ref{KG and Subspaces}, we have $K_G(\ell_\infty^n,F)\geqslant K_G(\ell_\infty^2,F).$ Since $\ell_\infty^2$ is isometric to $\ell_1^2$, using Lemma \ref{basic-properties-KG}, we readily have $K_G(\ell_\infty^2,F)=\rho(\ell_\infty^2,F).$ 
Now the result follows from \cite[Proposition 14]{XOR games and GT}.
\end{proof}
In the previous arXiv version, part $(ii)$ of the following proposition incorrectly asserted that \(\rho(\ell_1^n,\ell_p^n)=n^{1/p}\) for \(1<p\le2\). The correct asymptotic order is \(n^{1/2}\), as follows from \cite[Proposition 3]{BDPR}  and the subsequent remark.
\begin{proposition}
Let $p\geqslant 1,$  
then, the following statements are true
\begin{itemize}
\item[(i)] for $2\leqslant p<\infty$,   $\rho(\ell_1^n,\ell_p^n)=n^{\frac{1}{p^\prime}},$ and 
\item[(ii)] for $1<p\leqslant 2$, $\rho(\ell_1^n,\ell_p^n)=n^{\frac{1}{2}}.$
\end{itemize}
\end{proposition}
\begin{proof}
Note that $\rho(\ell_1^n,\ell_2^n)=\sqrt{n}$ \cite[Equation (59)]{XOR games and GT}. Moreover, by \cite[Proposition 12]{XOR games and GT}, it follows that 
\[\rho(\ell_1^n,\ell_p^n)\leqslant \rho(\ell_1^n,\ell_2^n)d(\ell_2^n,\ell_p^n) \quad \forall p\geqslant 1.\] 
For $p$ such that $2\leqslant p < \infty$,
it is well known that $d(\ell_2^n,\ell_p^n)=n^{\frac{1}{2}-\frac{1}{p}}$ (see \cite{HandbookI}). Thus, it follows that $\rho(\ell_1^n,\ell_p^n)\leqslant n^{\frac{1}{p^\prime}}.$ Note that for a tensor of the form $z=\sum_{i=1}^nz_i\mathbf{e}_i\otimes\mathbf{e}_i$ we have $\|z\|_{\ell_1^n\check{\otimes}\ell_p^n}=\|(z_i)_{i=1}^n\|_p$ and $\|z\|_{\ell_1^n\hat{\otimes}\ell_p^n}=\|(z_i)_{i=1}^n\|_1.$ Thus we have $\rho(\ell_1^n,\ell_p^n)\geqslant n^{\frac{1}{p^\prime}}.$ This completes the proof of part(i) of the proposition. We omit the proof of part $(ii)$ since the proof is available in \cite{BDPR}. 
\end{proof}

\begin{lemma}\label{applyncLpGT}
Let $1<p<2$ and $z\in S_p^{n,sa}\check{\otimes}S_p^{n,sa}.$ Suppose  $\widetilde{z}:S_{p'}^{n,sa}\to S_p^{n,sa}$ is the corresponding linear map associated with $z$. Then there exist a constant $K_p$ and a positive linear functional $\phi$ on $S_{p'/2}^n$ of unit norm such that 
\[\|\widetilde{z}(x)\|_{S_p^{n,sa}}\leqslant K_p\|z\|_{S_p^{n,sa}\check{\otimes}S_p^{n,sa}}(\phi(x^2))^{\frac{1}{2}}, \quad \forall x\in S_{p'}^{n,sa}.\] 
The constant $K_p$ depends on the cotype constants of $S_{p}$ and $S_{p'}.$
\end{lemma}
\begin{proof}
By considering $z$ to be also a tensor in $S_p^n\check{\otimes}S_p^n$, we claim that 
\begin{eqnarray}\label{comparing norm of self-adjoint tensor}
    \|z\|_{S_p^{n,sa}\check{\otimes}S_p^{n,sa}}\leqslant \|z\|_{S_p^{n}\check{\otimes}S_p^{n}}\leqslant 4\|z\|_{S_p^{n,{sa}}\check{\otimes}S_p^{n,sa}}.
\end{eqnarray}
The first inequality in \eqref{comparing norm of self-adjoint tensor} is trivial. For the second inequality in \eqref{comparing norm of self-adjoint tensor}, note that \[\|z\|_{S_p^{n,{sa}}\check{\otimes}S_p^{n,sa}}=\sup\left\{|\operatorname{tr}((a\otimes b)z)|:a,b\in S_{p'}^{n,sa},\ \|a\|_{S_{p'}^{n, sa}}=\|b\|_{S_{p'}^{n, sa}}=1\right\}.\] 
In a similar way we have
\[\|z\|_{S_p^{n}\check{\otimes}S_p^{n}}=\sup\left\{|\operatorname{tr}((a\otimes b)z)|:a,b\in S_{p'}^{n},\ \|a\|_{S_{p'}^{n}}=\|b\|_{S_{p'}^{n}}=1\right\}.\]
Also,  note that for $a,b\in S_{p'}^{n},$ we have
\begin{multline*}
|\operatorname{tr}((a\otimes b)z)|\leqslant \\|\operatorname{tr}((\text{Re}\, a\otimes\text{Re}\, b)z)|+|\operatorname{tr}((\text{Re}\, a\otimes\text{Im}\, b)z)|+|\operatorname{tr}((\text{Im}\, a\otimes\text{Re}\,b)z)|+|\operatorname{tr}((\text{Im}\, a\otimes\text{Im}\, b)z)|.
\end{multline*}
 Moreover, $\|a\|_{S_{p'}^n}\leqslant 1$ implies $\|\text{Re}\, a\|_{S_{p'}^n}\leqslant 1$ and $\|\text{Im}\,a\|_{S_{p'}^n}\leqslant 1.$ Therefore, we readily have 
\[\|z\|_{S_p^{n}\check{\otimes}S_p^{n}}\leqslant 4\|z\|_{S_p^{n,{sa}}\check{\otimes}S_p^{n,sa}}.\] 
This proves the claim. 
Consider the bilinear form $B_z:S_{p'}^n\times S_{p'}^{n}\to\mathbb C$ defined by  $B_z(x,y):=\operatorname{tr}((x\otimes y)z).$ 
Then \[\|B_z\|=\|z\|_{S_p^{n}\check{\otimes}S_p^{n}}.\] 
By Theorem \ref{LpncG}, there exist positive unit norm linear functionals $\phi$ and $\psi$ on $S_{p'/2}^n$ such that 
\[|\operatorname{tr}(\widetilde{z}(x)y)|=|\operatorname{tr}((x\otimes y)z)|\leqslant K_p\|z\|_{S_p^{n}\check{\otimes}S_p^{n}}\bigg(\phi\Big(\frac{x^*x+xx^*}{2}\Big)\bigg)^{\frac{1}{2}}\bigg(\psi\Big(\frac{y^*y+yy^*}{2}\Big)\bigg)^{\frac{1}{2}}\] for all $x,y\in S_{p'/2}^n.$ 
Note that if $y\in S_{p'}^{n,sa}$ with $\|y\|_{S_{p'}^{n,sa}}=1,$ then $\|y^*y\|_{S_{p'/2}^{n}}=\|yy^*\|_{S_{p'/2}^{n}}=1.$ 
Thus, by taking supremum over $\|y\|_{S_{p'}^{n,sa}}=1$, we obtain the desired result.
\end{proof}

 \begin{proposition}\label{geqnclp}
     For  $1< p<2$ and $n\in\mathbb{N}$, there is a universal constant $C$ independent of $n$ such that  $\rho(S_p^{n,\text{sa}},S_p^{n,\text{sa}})\geqslant Cn^{\frac{5}{2}-\frac{1}{p}}.$
 \end{proposition}
 \begin{proof}
     We follow the strategy of \cite[page 702]{XOR games and GT}. Let us fix an orthonormal basis $(x_i)_{i=1}^{n^2}$ in $M_n^{\text{sa}}$, the space of all self-adjoint matrices of order $n,$ with respect to the Hilbert-Schmidt norm. Let us consider the random tensor $z=\sum_{i,j=1}^{n^2}g_{ij}x_i\otimes x_j$ where $g_{ij}$'s are iid $N(0,1).$ Then by Chevet's theorem (Theorem \ref{chevet}), we have that 
     \[
\mathbb{E}\|z\|_{S_p^{n,\text{sa}}\check{\otimes}S_p^{n,\text{sa}}}\leqslant 2\|\text{id}\|_{S_2^{n,\text{sa}}\to S_p^{n,\text{sa}}}\mathbb{E}\Big\|\sum_{i=1}^{n^2}g_ix_i\Big\|_{S_p^{n,\text{sa}}}.
     \] 
     It can be deduced from \cite[Theorem 1.3]{Szarekcond}, $\mathbb{E}\Big\|\sum_{i=1}^{n^2}g_ix_i\Big\|_{S_p^{n,\text{sa}}}\leqslant C n^{\frac{1}{2}+\frac{1}{p}}$ for some universal constant $C>0$. It is easy to see that \[
        \|\text{id}\|_{S_2^{n,\text{sa}}\to S_p^{n,\text{sa}}} = \begin{cases}
                        1 , \ \text{for}\ p\geqslant 2; \\
                        n^{\frac{1}{p}-\frac{1}{2}}, \ \text{for}\ p<2.
                    \end{cases}
\]
Therefore, we have that 
    \[
       \mathbb{E}\|z\|_{S_p^{n,\text{sa}}\check{\otimes}S_p^{n,\text{sa}}}\leqslant C \begin{cases}
                        n^{\frac{1}{2}+\frac{1}{p}} , \text{for}\ p\geqslant 2; \\
                        n^{\frac{2}{p}}, \ \text{for}\ p<2.
                    \end{cases}
\]
     Note that by proceeding, as in \cite{XOR games and GT}, we have the following inequality
     \[\mathbb{E}\bigg(\sum_{i,j=1}^{n^2}g_{ij}^2\bigg)^{\frac{1}{2}}\leqslant \sqrt{\mathbb{E}\|z\|_{S_p^{n,\text{sa}}\hat{\otimes}S_p^{n,\text{sa}}}}\sqrt{\mathbb{E}\|z\|_{S_{p^\prime}^{n,\text{sa}}\check{\otimes}S_{p^\prime}^{n,\text{sa}}}},\] 
     Therefore, we have that $\mathbb E\|z\|_{S_p^{n,\text{sa}}\hat{\otimes} S_p^{n,\text{sa}}}\geqslant cn^{\frac{7}{2}-\frac{1}{p^\prime}}.$ Arguing as in \cite{XOR games and GT}, we have 
     \[\rho(S_p^{n,\text{sa}},S_p^{n,\text{sa}})\geqslant Cn^{\frac{5}{2}-\frac{1}{p}}.\]
     This completes the proof of the proposition.
 \end{proof}
 \begin{proposition}\label{leqnclp}
      For $1<p <2$, there exists a positive constant $K$ such that     
      $$\rho(S_p^{n,\text{sa}},S_p^{n,\text{sa}})\leqslant Kn^{\frac{3}{2}+\frac{1}{p'}} \quad \forall n\geqslant 1.$$ 
 \end{proposition}
 \begin{proof}
 The proof is similar to that of Theorem 8 in \cite{XOR games and GT}.
     Let $z\in M_n^{\text{sa}}\otimes M_n^{\text{sa}}$ be such that $\|z\|_{S_p^{n,\text{sa}}\check{\otimes} S_p^{n,\text{sa}}}=1.$ By Lemma \ref{applyncLpGT}, there exists a positive linear functional $\varphi$ on $S_{p'/2}^n$ such that 
     \[\|\widetilde{z}(x)\|_{S_p^{n,sa}}\leqslant K(\phi(x^2))^{\frac{1}{2}},\quad x\in S_{p'}^{n,sa},\]
      where $\widetilde z:S_{p'}^{n,\text{sa}}\to S_p^{n,\text{sa}}$ is the realization of $z$ as a linear map. The constant $K$ depends on the cotype constants of $S_{p}$ and $S_{p'}.$ 
      There exists an orthonormal basis $(u_j)_{j=1}^n$ such that we have $\varphi=\sum_{j=1}^n\lambda_jP_j$ where $P_j(u):=\langle u,u_j\rangle u_j$ for $1\leqslant j\leqslant n.$ Now define $E_{jk}$ as $E_{jk}(u):=\langle u,u_k\rangle u_j$ for all $1\leqslant j,k\leqslant n.$ Then it is easy to check that $((E_{jk})_{j,k=1}^n,(E_{kj})_{k,j=1}^n)$ is a biorthogonal system. By denoting $F_{jk}=E_{jk}+E_{kj}$ and $H_{jk}=i(E_{jk}-E_{kj})$, we have 
     \[z=\sum_{j=1}^nE_{jj}\otimes \widetilde{z}(E_{jj})+\frac{1}{2}\sum_{j<k}\Big(F_{jk}\otimes\widetilde{z}(F_{jk})+H_{jk}\otimes\widetilde{z}(H_{jk})\Big).\] 
     Again proceeding as in \cite{XOR games and GT}, we must have that
\[\|z\|_{S_p^{n,\text{sa}}\hat{\otimes}S_p^{n,\text{sa}}}\leqslant K\bigg(\sum_{j=1}^n\sqrt{\lambda_j}+\sum_{1\leqslant j<k\leqslant n}\sqrt{\lambda_j+\lambda_k}\bigg).\] 
Note that we must have $\sum_{j=1}^n\lambda_j^{\big(\frac{p^\prime}{2}\big)'}=1.$  
Hence, it follows that $\|z\|_{S_p^{n,\text{sa}}\hat{\otimes}S_p^{n,\text{sa}}}\leqslant Kn ^{\frac{3}{2}+\frac{1}{p'}}.$
This completes the proof of the proposition.
\end{proof}
Using Corollary \ref{impcor}, Proposition \ref{geqnclp} and Proposition \ref{leqnclp}, we have the following corollary.
 \begin{corollary}
    For $1\leqslant p<2$, there exist two positive constants $c$ and $C,$ depending only on $p,$ such that 
    \[cn^{\frac{5}{2}-\frac{1}{p}}\leqslant K_G(S_p^{n,\text{sa}},S_p^{n,\text{sa}})\leqslant Cn^{\frac{5}{2}-\frac{1}{p}}, \quad \forall n\geqslant 1.\] 
 \end{corollary}
The following two propositions give an alternative proof of the diagonal estimate obtained in \cite[ Proposition 3]{BDPR} and the subsequent remark. We retain these proofs because they parallel the preceding arguments for Schatten classes. For estimates involving two independently varying dimensions, see \cite{NPS}.
 \begin{proposition}
      If $1< p<2$ and $n\in\mathbb{N}$, then $\rho(\ell_p^{n},\ell_p^{n})\geqslant Cn^{\frac{3}{2}-\frac{1}{p}}$ for some constant $C$ independent of $n$.
 \end{proposition}
 \begin{proof}
     Consider a random tensor of the form $z=\sum_{i,j==1}^{n}g_{ij}\mathbf{e}_i\otimes\mathbf{e}_j$ where $g_{ij}$'s are iid $N(0,1).$ 
     By Kahane's inequality, there exists a constant $C_p>0,$ such that
     \[\mathbb{E}\Big\|\sum_{i=1}^{n}g_i\mathbf{e}_i\Big\|_{\ell_p^{n}}\leqslant C_p\Big(\mathbb{E}\Big\|\sum_{i=1}^{n}g_i\mathbf{e}_i\Big\|_{\ell_p^{n}}^p\Big)^{\frac{1}{p}}=C_p\Big(\sum_{i=1}^{n}\mathbb{E}|g_i|^p\Big)^{\frac{1}{p}}=D_pn^\frac{1}{p}.\] 
     In above the last equality, $D_p=(\mathbb{E}|g|^p)^{\frac{1}{p}}\simeq (\int_{0}^\infty t^pe^{-t^2}dt)^{\frac{1}{p}}<\infty.$
Then proceeding as in Proposition \ref{geqnclp}, by Chevet's theorem (Theorem \ref{chevet}), we obtain that 
    \[\mathbb{E}\|z\|_{\ell_p^{n}\check{\otimes}\ell_p^{n}}\leqslant \begin{cases}
                        D_p n^{\frac{1}{p}} , \text{for}\ p\geqslant 2; \\
                        D_p n^{\frac{2}{p}-\frac{1}{2}}, \ \text{for}\ p<2.
                    \end{cases}\]
     Note that by proceeding as in Proposition \ref{geqnclp}, we have the following inequality
     \[\mathbb{E}\Big(\sum_{i,j=1}^ng_{ij}^2\Big)^{\frac{1}{2}}\leqslant \sqrt{\mathbb{E}\|z\|_{\ell_p^{n}\hat{\otimes}\ell_p^{n}}}\sqrt{\mathbb{E}\|z\|_{\ell_{p^\prime}^{n}\check{\otimes}\ell_{p^\prime}^{n}}}.\] 
     Therefore, we get that $\mathbb E\|z\|_{\ell_p^{n}\hat{\otimes} \ell_p^{n}}\geqslant cn^{{2}-\frac{1}{p^\prime}}.$ Arguing as before, we have $\rho(\ell_p^{n},\ell_p^{n})\geqslant Cn^{\frac{3}{2}-\frac{1}{p}}.$
     This completes the proof of the proposition.
 \end{proof}
  \begin{proposition}
     For $1<p < 2$, there exists a constant $K>0$ such that
     $$\rho(\ell_p^{n},\ell_p^{n})\leqslant Kn^{\frac{1}{2}+\frac{1}{p'}}, \quad \forall n\geqslant 1.$$ 
 \end{proposition}
 \begin{proof}
     Let $z\in \ell_p^{n}\otimes\ell_p^{n} $ be such that $\|z\|_{\ell_p^{n}\check{\otimes}\ell_p^{n}}=1.$ Suppose $\widetilde z:\ell_{p'}^{n}\to\ell_p^{n}$ denotes the linear map associated with $z.$ By $L_p$-Grothendieck theorem \cite{Maurey}, there exists a positive linear functional $\varphi$ on $\ell_{p'/2}^n$ such that  
     $$\|\widetilde{z}(x)\|_{\ell_p^{n}}\leqslant K(\phi(x^2))^{\frac{1}{2}}$$ 
     for all $x\in\ell_{p'}^{n}.$ The constant $K$ depends on the cotype constants of $\ell_{p}$ and $\ell_{p'}.$ By identifying $\varphi$ with  $\sum_{j=1}^n\lambda_j\mathbf{e}_j,$ where $
(\mathbf{e}_j)_{j=1}^n$ is standard basis, we have $z=\sum_{i=1}^n\mathbf{e}_i\otimes \widetilde{z}(\mathbf{e}_i),$ since $(\mathbf{e}_j)_{j=1}^n$ is a biorthogonal system. Thus, we obtain the estimate
\[\|z\|_{\ell_p^{n}\hat{\otimes}\ell_p^{n}}\leqslant K\Big(\sum_{j=1}^n\sqrt{\lambda_j}\Big).\] 
Note that we must have $\sum_{j=1}^n\lambda_j^{\big(\frac{p^\prime}{2}\big)'}=1.$  Therefore, we must have
$\|z\|_{\ell_p^{n}\hat{\otimes}\ell_p^{n}}\leqslant Kn ^{\frac{1}{2}+\frac{1}{p'}}.$
This completes the proof.
\end{proof}
 \begin{corollary}
    For $1\leqslant p<2$, there exist positive constants $c$ and $C$, depending only on $p,$ such that 
    \[cn^{\frac{3}{2}-\frac{1}{p}}\leqslant K_G(\ell_p^{n},\ell_p^{n})\leqslant Cn^{\frac{3}{2}-\frac{1}{p}},\quad \forall n\geqslant 1.\]
 \end{corollary}
\section{The positive Grothendieck constant $K^{+}_G(E,F)$} \label{S5}
A finite-dimensional Banach space $E$ may possess several interesting properties relative to taking different tensor products of $E$ with itself. Two such properties were studied in \cite{BM}. 

First, we introduce a quantitative variant of ``Property P" introduced in \cite{BM}. We use the natural notion of positivity for elements of $E \otimes E$, namely, an element $A$ of $E \otimes E$ is positive (written $A \geq 0$ ) if it is in the convex hull of the set of symmetric tensors $x \otimes x, x \in E$. Given a finite-dimensional complex Banach space $E$, let 
\[\gamma(E):= \inf \{c: \langle A, B\rangle_{\text{HS}} \leqslant c\|A\|_{E\widecheck{\otimes}E} \| B \|_{E^*\widecheck{\otimes} E^*},\,\, A \in E \otimes E,\,B \in E^* \otimes E^*,A \geqslant 0, B \geqslant 0\},\] where $\langle \cdot, \cdot \rangle_{\text{HS}}$ is the Hilbert-Schmidt inner product. Then, in the terminology of \cite{BM}, $E$ is said to possess ``Property P'' if $\gamma(E) \leq 1$. 

We now connect the constant $\gamma(E)$ to the positive Grothendieck constant $K_G^+(E,F)$ defined in \eqref{Definition-KG+}. The proof of the third equality in the Proposition below is essentially the same as that of \cite[Fact 2]{BM}. It is reproduced here for the sake of completeness.

\begin{proposition} \label{++++}
For any finite-dimensional Banach space $E$, we have \[\gamma(E)=K_G^+(E,\ell_2^{\text{dim}\,E})=\sup\{K_G^+(E,\ell_2^n):n\geqslant 1\}=\|id\|^2_{\Pi_2(E,\ell_2)\to B(E,\ell_2)}.\]
\end{proposition}
\begin{proof}
Suppose $\dim E=m.$  Let $A:E\to E^*$ and $B:E^*\to E$ be positive operators. Note that $B=C^*C$ for some linear map $C:E^*\to \ell_2^m$ and $A=\beta^*\beta$ for some $\beta:E\to\ell_2^m.$ it follows that 
\begin{eqnarray*}
|\text{tr}(BA)| &=& |\text{tr}(C^*CA)|\leqslant  N(C^*CA)\leqslant  
\|C\|^2 \|A\| K_G^+(E,\ell_2^m)=\|A\| \|B\| K_G^+(E,\ell_2^m).
\end{eqnarray*}
Thus we have $\gamma(E)\leqslant K_G^+(E,\ell_2^m).$
For the other side inequality, 
note that
\begin{align*}
\sup\big \{N(BA):\|B\|_{E^*\to \ell_2^m}\leqslant 1 \big\}&= \sup\big \{ |\text{tr}(CB\beta^*\beta)| : \|B\|_{E^*\to \ell_2^m}\leqslant 1,\, \|C\|_{\ell_2^m \to E}\leqslant 1 \big \}\\
&=\sup\big \{  |\text{tr}(B\beta^*\beta C)|  : \|B\|_{E^*\to \ell_2^m}\leqslant 1,\, \|C\|_{\ell_2^m \to E}\leqslant 1 \big \}\\
&=  \sup\big \{ |\text{tr}((\beta B^*)^*(\beta C))| : \|B\|_{E^*\to \ell_2^m}\leqslant 1,\, \|C\|_{\ell_2^m \to E}\leqslant 1 \big \}\\ 
&\leqslant   \sup\big \{  \|B\beta^*\|_2\|\beta C\|_2 : \|B\|_{E^*\to \ell_2^m}\leqslant 1,\, \|C\|_{\ell_2^m \to E}\leqslant 1 \big \}\\ 
&=\sup \big \{\|B\beta^*\|_2^2 :\|B\|_{E^*\to \ell_2^m}\leqslant 1\big\}\\
&= \sup\{|\text{tr}(DA)|: D\geqslant 0,\ \|D\|_{E^* \to E}\leqslant 1\}.
\end{align*} 
Now, taking supremum over all positive operators $A:E\to E^*,$ with $\|A\|_{E\to E^*}\leq 1,$ we get $K_G^+(E,\ell_2^m)\leqslant \gamma(E).$ This completes the proof of the first equality in the proposition. 

 The second equality is evident.

To verify the third equality, let $x_1,\dots,x_k\in E$ and $u\in B(E,\ell_2).$ Define an operator $T:\ell_2^k\to E$ by $Te_i=x_i,$ where $1\leqslant i\leqslant k$ and $e_i$'s are canonical basis of $\ell_2^k$. Note that $T^*(x^*)=\sum_{i=1}^kx^*(x_i)e_i$ for any $x^*\in E^*.$ 
Therefore, $\|T\|_{\ell_2^k\to E}=\|T^*\|_{E^*\to\ell_2^k}\leqslant 1$ is equivalent to the condition $\sum_{i=1}^k|x^*(x_i)|^2\leqslant 1$ for all $x^*\in (E^*)_1.$  
Also, note that \[\sum_{i=1}^k\|u(x_i)\|_{2}^2=\sum_{i=1}^k\|uT(e_i)\|_{2}^2=\langle u^*u, TT^*\rangle_{HS}.\] 
Therefore, by denoting $S=TT^*,$ we conclude that 
\[\pi_2(u)^2=\sup\big\{\langle u^*u, S\rangle: S\geqslant 0,\, \|S\|_{E^*\to E}\leqslant 1\big\}.\] 
Hence we get $\gamma(E)=\|id\|^2_{\Pi_2(E,\ell_2)\to B(E,\ell_2)}.$
\end{proof}
Recall that positive  variant of the Grothendieck inequality asserts that $\sup_{n\in\mathbb N}\gamma(\ell_\infty^n)<\infty.$ In fact, $\sup_{n\in\mathbb N}\gamma(\ell_\infty^n)=\frac{\pi}{2}.$ This motivates the following definition.
\begin{definition}[Little G. T. flag]
Let $(E_n)_{n\in\mathbb N}$ be a sequence of finite-dimensional Banach spaces with the property that
\begin{itemize}
    \item[H1:] $\dim(E_n)=n$ for all $n\in\mathbb N.$ 
    \item[H2:] there exists an isometry $j_n:E_n\to E_{n+1}$ for each $n\in\mathbb N$. 
\end{itemize}
We say $(E_n,j_n)_{n\in\mathbb N}$ is a Little G.T. flag if 
$\sup\{\gamma(E_n): n\in\mathbb N\}$ is finite.
\end{definition}
\begin{remark}\label{normalin}
Let $(E_n)_{n\geqslant 1}$ be a sequence of Banach spaces 
satisfying (H1) and (H2).
We discuss the inductive limit of $(j_n,E_n)_{n\geqslant 1}$: Consider the subspace of $\prod_{n\geqslant 1}E_n$ formed by sequences $(x_n)_{n\geqslant 1}$ with $j_nx_n=x_{n+1}$ for all $n$ large. We can set $\|x\|:=\lim\limits_{n\to\infty}\|x_n\|_{E_n}.$ 
Clearly, this defines a seminorm. After taking quotient by the subspace  $\{x:\|x\|=0\}$ and taking closure, we obtain a Banach space which is denoted by $E^{\text{ind}}.$ We have a canonical isometric inclusion of $E_n$ into $E^{\text{ind}}$ for all $n\geqslant 1.$ Under this identification, we may assume that $E_n\subseteq E_{n+1}$ for all $n\in\mathbb N$ and $E^{\text{ind}}=\overline{\cup_{n=1}^\infty E_n}.$ Without loss of generality, from now on we may assume that the maps $j_n$'s are all inclusion maps. 
\end{remark}
\begin{definition}[Hilbert-Schmidt space, \cite{Jarchow82}]
A Banach space $E$ is called a Hilbert-Schmidt space if any bounded linear operator $u:E\to\ell_2$ is a $2$-summing operator, i.e., the identity operator $\operatorname{id}:\Pi_2(E,\ell_2)\to B(E,\ell_2)$ is an isomorphism.
\end{definition}
The proof of the following proposition is an adaption of \cite[Fact 2]{BM}.
\begin{proposition} \label{prop:5.5}
Let $E$ be a Banach space and $(E_n)_{n\in\mathbb{N}}$ be a sequence of Banach spaces satisfying $E_n\subseteq E_{n+1}$ for all $n\geq 1$ and $E=\overline{\cup_{n=1}^\infty E_n}.$ Moreover, assume that for all $n\geq 1$ there are projections $P_n:E\to E_n\subseteq E$ with $\sup_{n\geq 1}\|P_n\|<\infty.$ Then  $(E_n,j_n)_{n\in\mathbb N}$ is a Little G.T. flag if and only if $E$ is a Hilbert-Schmidt space where $j_n:E_n\to E_{n+1}$ is the canonical inclusion map.
\end{proposition}
\begin{proof}
We first prove the `if' part. Note that for $u\in B(E_n,\ell_2),$ we can define a bounded linear operator $\tilde{u}:E\to\ell_2$  as $\tilde{u}=uP_n.$ Using \cite[Chapter 1, Page 9]{Pisier-factorization} and the hypothesis that $E$ is a Hilbert-Schmidt space, we get a constant $C$ (independent of $n$ and $u$) such that $\pi_2(u)\leq\pi_2(\tilde{u})\leq C\|u\|_{E_n\to\ell_2}.$ 
By Proposition \ref{++++}, we get
 $\gamma(E_n)\leq \sqrt{C}$ for all $n\geq 1.$

To prove the `only if' part we first observe that from hypothesis that $(E_n,j_n)_{n\in\mathbb N}$ is a Little G.T. flag and Proposition \ref{++++}, there exists a constant $A>0$ (independent of $n$) such that $\pi_2(v)\leq A\|v\|$ for all $v\in B(E_n,\ell_2).$ We now show that $\pi_2(u)\lesssim\|u\|$ for all $u\in B(E,\ell_2).$ We will be done by proving $\pi_2(u)\lesssim \sup_{n\geq 1}\pi_2(u{|}_{E_n})\lesssim\sup_{n\geq 1}\|u{|}_{E_n}\|\lesssim\|u\|_E.$ Note that in view of the previous estimate, we only need to prove that $\pi_2(u)\lesssim \sup_{n\geq 1}\pi_2(u{|}_{E_n}).$ Take $x_1,\dots,x_k\in E$ and consider $a_1^n,\dots,a_k^n\in E_n$ such that $\|x_i-a_i^n\|<\frac{1}{n}$ for all $1\leq i\leq k.$  We have that 
\begin{equation}\label{pi2ine}
    \Big(\sum_{i=1}^k\|u(a_i^n)\|^2\Big)^{\frac{1}{2}}\leq \sup_{n\geq 1}\pi_2(u{|}_{E_n})\sup\Big\{\Big(\sum_{i=1}^n|x^*(a_i^n)|^2\Big)^{\frac{1}{2}}:\|x^*\|_{E_n^*}\leq 1\Big\}.
\end{equation}
By the Hahn-Banach theorem, one can see that
\[\sup\Big\{\Big(\sum_{i=1}^n|x^*(a_i^n)|^2\Big)^{\frac{1}{2}}:\|x^*\|_{E_n^*}\leq 1\Big\}=\sup\Big\{\Big(\sum_{i=1}^n|x^*(a_i^n)|^2\Big)^{\frac{1}{2}}:\|x^*\|_{E^*}\leq 1\Big\}.\]
Hence the desired inequality follows if we replace right-hand side of inequality \eqref{pi2ine} with the above equality and taking $n\to\infty.$ This completes the proof of the lemma.
\end{proof}
\begin{remark}\label{nxampl}
Note that the proof of the above proposition reveals that  the inductive limit of a sequence of G.T. flag has to be a Hilbert-Schmidt space. A large class of examples of Hilbert-Schmidt spaces are mentioned in \cite{Jarchow82}.
\end{remark}
\begin{lemma}
For any Banach space $F$, we have $K_G^+(\ell_2^n,F)=\rho(\ell_2^n,F).$
\end{lemma}
\begin{proof}
As the supremum in the extremum of $K_G^+(\ell_2^n,F)$ is attained at the identity operator, which is also a positive operator, it follows that $K_G^+(\ell_2^n,F)=\rho(\ell_2^n,F).$
\end{proof}
For a finite-dimensional Banach space $E,$ we define \[\rho^{+}(E):=\sup\{\|z\|_{E\hat{\otimes}E}:\|z\|_{E\check{\otimes}E}\leqslant 1,\ z\geqslant 0\}.\]

\begin{lemma}
$K_G^+(\ell_\infty^n,\ell_\infty^n)\geqslant \rho^+(\ell_1^n).$
\end{lemma}
\begin{proof}
In view of Lemma \ref{Pisier-Nuclear}, taking $B=\operatorname{id}$ in \eqref{Nuclear}, we get,
\begin{eqnarray*}
K_G^+(\ell_\infty^n,\ell_\infty^n)= K_G^+(\ell_\infty^n,\ell_1^n)&=& \sup \{N(BA^*): A\geqslant 0,\, \|A\|_{\ell_\infty^n\to \ell_1^n}\leqslant 1,\, \|B\|_{\ell_1^n \to\ell_1^n}\leqslant 1\}\\
&\geqslant &\sup\{\|A\|_{\ell_1^n\hat{\otimes}\ell_1^n}: \|A\|_{\ell_\infty^n \to\ell_1^n}\leqslant 1, \ A\geqslant 0\}\\
&=&\rho^{+}(\ell_1^n).
\end{eqnarray*}
This completes the proof of the lemma.
\end{proof}
\begin{theorem}\label{rho+-1-1}
There exists a constant $c>0$ such that for large $n,$ $\rho^+(\ell_1^n)\geqslant c\sqrt{n}.$
\end{theorem}
\begin{proof}
Consider the sequence of matrices $B_n=(b_{jk})_{j,k=0}^{n-1}$ with 
\begin{eqnarray*}
b_{jk}=\sqrt{\frac{1}{n}}cos \Big(\frac{2\pi jk}{n}\Big), \quad j,k=0,\ldots, n-1.
\end{eqnarray*}  
Since $B_n$ is nothing but the real part of discrete Fourier transform matrix, the operator $B_n:\ell_2^n\to \ell_2^n$ has norm atmost $1$ for all $n\in\mathbb N.$ It follows that $B_n:\ell_\infty^n \to \ell_1^n$ is of norm atmost $n$ for each $n\in\mathbb N$. 
Moreover, for large $n$, the quantity $\sum_{j,k=1}^{n}|b_{jk}|$ is of order $n^{3/2}$ \cite{Pinelis}. 
Since $B_n$ is real symmetric matrix and is an operator on $\ell_2^n$ of norm atmost $1$, it follows that spectrum of $B_n$ is contained in $[-1,1].$ Thus $A_n:=B_n+I$ is a positive operator and 
$$\|A_n\|_{\ell_\infty^n\to \ell_1^n}\leqslant \|B_n\|_{\ell_\infty^n\to \ell_1^n}+\|I\|_{\ell_\infty^n\to \ell_1^n}\leqslant 2n.$$  
On the other hand $\sum_{j,k=0}^{n-1}|a_{jk}|$ is still at least of order $n^{3/2},$ where $a_{jk}$ is the $(j,k)$ entry of $A_n$. Choosing $A_n$ as above, by definition,  we have
\begin{eqnarray*}
\rho^+(\ell_1^n) &=& \sup_{X\geqslant 0}\frac{\|X\|_{\ell_1^n\hat{\otimes}\ell_1^n}}{\|X\|_{\ell_1^n\check{\otimes}\ell_1^n}}\\
&=& \sup_{X\geqslant 0}\frac{\|X\|_{\ell_1^n\hat{\otimes}\ell_1^n}}{\|X\|_{\ell_\infty^n\to \ell_1^n}}\\
&\geqslant & \frac{\sum_{j,k=0}^{n-1} |a_{jk}|}{\|A\|_{\ell_\infty^n\to \ell_1^n}}\\
&\geqslant & \frac{cn^{3/2}}{n}\\
&=& c\sqrt{n}.
\end{eqnarray*}
for some positive constant $c$ independent of $n.$
This completes the proof of the lemma.
\end{proof}
\begin{remark}
The example of matrix $B_n$ also plays a vital role in the context of Grothendieck inequality and von Neumann inequality (see \cite{Blei} and \cite{Ton78}). Some of the properties of this matrix that we have used are also in \cite{Ton78}. In fact, it is not hard to see from \cite{Ton78} that the absolute sum of coefficients of either real or imaginary part of discrete Fourier transform matrix is of order $n^{3/2}.$ this information also suffices for the above proof to make work.
\end{remark}
The verification of the useful characterization of a non-negative contraction $A$ from $\ell_1^n$ to $\ell_\infty^n$, given below, follows from the observation that $\|A\|_{\ell_1^n\to\ell_\infty^n}=\max_{1\leqslant i,j\leqslant n}|a_{ij}|$. 
\begin{fact}\label{fact6.7}   Suppose that $A$ is a nonnegative linear map from $\ell_1^n$ to $\ell_\infty^n$. Then $\|A\|_{\ell_1^n\to\ell_\infty^n}\leqslant 1$ if and only if  there exists a finite-dimensional Hilbert space $\mathcal H$ and $v_1,\dots,v_n\in\mathcal H$ with $\|v_i\|_{\mathcal H}\leqslant 1$ for all $1\leqslant i\leqslant n$ such that $a_{ij}=\langle v_i,v_j\rangle$ for all $1\leqslant i\leqslant n.$
\end{fact}

\begin{proposition}\label{positive-1-infinity}
For all $n\geqslant 1$, we have $K_G^{+}(\ell_1^n,\ell_1^n)=K_G^{+}(\ell_1^n,\ell_\infty^n)\leqslant K_G.$
\end{proposition}
\begin{proof}
The first equality we have to prove is the positive variant of the equality in Lemma \ref{basic-properties-KG}(i) and is easy to verify.  Therefore, we only verify $K_G^{+} (\ell_1^n,\ell_\infty^n)\leqslant K_G$ below.

Appealing to Lemma \ref{Pisier-Nuclear}, we have
\begin{eqnarray*}
K_G^+(\ell_1^n,\ell_\infty^n) &=& \sup \{N(XA^*): A\geqslant 0,\, \|A\|_{\ell_1^n\to \ell_\infty^n}\leqslant 1,\, \|X\|_{\ell_\infty^n \to\ell_\infty^n}\leqslant 1\}\\\nonumber
&= &\sup\{|\operatorname{tr}(CXA^*)|: A\geqslant 0,\, \|A\|_{\ell_1^n\to \ell_\infty^n}\leqslant 1,\, \|X\|_{\ell_\infty^n \to\ell_\infty^n}\leqslant 1,\ \|C\|_{\ell_\infty^n\to\ell_1^n}\leqslant 1\}\\\nonumber
&=  &\sup\{|\operatorname{tr}(ZA^*)|:A\geqslant 0,\, \|A\|_{\ell_1^n\to \ell_\infty^n}\leqslant 1,\, \|Z\|_{\ell_\infty^n \to\ell_1^n}\leqslant 1\}\\\nonumber
&=&\sup\{|\sum_{i,j=1}^nz_{ij}\langle v_i,v_j\rangle|:\|Z\|_{\ell_\infty^n \to\ell_1^n}\leqslant 1,\ \|v_i\|_2\leqslant 1\}\leqslant K_G, 
\end{eqnarray*}
where $Z= \big (\!\!\big ( z_{ij}\big )\!\!\big )$. The last equality follows from Fact \ref{fact6.7}.
\end{proof}

\section{Applications}\label{adde}
In what follows, all Banach spaces are over the field of complex numbers. What is discussed below relates the positive Grothendieck constant to the complete norm bound of a certain class of homomorphisms (or, equivalently, a class of linear maps on finite-dimensional Banach spaces) first introduced by Parrott in \cite{SKP}. We denote the vector space of complex $n_1\times n_2$ matrices by $\mathbb{C}^{n_1\times n_2}.$

Let $\Omega$ be a bounded open connected subset of $\mathbb{C}^m$ and $H^\infty(\Omega)$ be the algebra of bounded holomorphic functions on $\Omega$. Suppose that $V:=(V_1, \ldots , V_m)$ is a $m$-tuple of linear transformations from $\mathbb{C}^{n_1}$ to $\mathbb{C}^{n_2}$, in other words, each $V_j$, $1\leq j \leq m$, is represented by a $n_1 \times n_2$ matrix, then $\varrho_V: H^\infty(\Omega) \to \mathbb{C}^{n \times n}$, where $n=n_1+n_2$, defined by 
\[\varrho_V(f):= \begin{pmatrix} f(w) I_{n_1} & Df(w) \cdot V \\ 0 & f(w) I_{n_2}\end{pmatrix},\,\,w\in \Omega,\,\, f\in H^\infty(\Omega), \]
where \[Df(w) \cdot V := \tfrac{\partial f}{\partial z_1}(w) V_1+ \cdots + \tfrac{\partial f}{\partial z_m}(w) V_m,\] 
is a homomorphism. The class of homomorphisms  
\[\{\varrho_V\mid V=(V_1, \ldots V_m), V_j\in \mathbb{C}^{n_1\times n_2}, 1\leq j \leq m\} \] was first introduced by Parrott in \cite{SKP}, and we call them \textit{Parrott} homomorphisms. 

Recall that the Carath\'{e}odory norm $C_{\Omega,w}(x)$ of a vector $x\in \mathbb{C}^m$ is defined to be   
\[\sup\{| Df(w)\cdot x| \mid f:\Omega \to \mathbb{D},\,\, f \text{ is holomorphic, }f(w)=0\}.\] 

We now assume that $\Omega$ is a ball with respect to some norm $\| \cdot \|$ and $w=0$. Then the set
\[\big \{\left(\partial_1 f(0), \partial_2 f(0), \ldots, \partial_m f(0)\right): f \in H^\infty(\Omega), \|f\|_\infty \leqslant 1, f(0)=0\big\}\]
coincides with the unit ball in the dual space $(\mathbb{C}^m, \|\cdot \|_\Omega)^*$,  where $\|\cdot\|_\Omega:= C_{\Omega,0}$. Consequently, 
\[\{x\in \mathbb{C}^m \mid C_{\Omega,0}(x) \leqslant 1\}=\Omega,\] see \cite[Lemma 2.2]{Misra19}. Let $E$ be the linear space $\mathbb{C}^m$ equipped with the dual of the norm $\|\cdot\|_\Omega$. Define 
\[L_V:(E, \|\cdot\|_\Omega^*) \to (\mathbb{C}^{n_1\times n_2}, \|\cdot \|_{\ell_2^{n_1} \to \ell_2^{n_2}})\] by setting 
\[L_V(x_1, \ldots , x_m)= x_1 V_1 + \cdots + x_m V_m,\,\, (x_1, \ldots , x_m)\in \mathbb{C}^m.\] 
Let  $E^{(k)}$ be the tensor product of $E$ with the space of $k\times k$ matrices  equipped with the usual operator norm. We choose 
the injective tensor product norm on $E^{(k)}$. Explicitly,  
if $X^{(k)}:=\big (\!\!\big ( x_{i,j} \big )\!\!\big ) \in  E^{(k)}$, then $X^{(k)}:(E, \|\cdot\|_\Omega) \to (\mathbb{C}^{k\times k}\|\cdot\|_{\ell_2^k \to \ell_2^k})$ and
\[ \|X^{(k)}\|=\sup \big \{\|\big (\!\!\big ( x_{i,j}(u)\big )\!\!\big )\|_{\ell^k_2 \to \ell^k_2}\mid u\in (E, \|\cdot\|_\Omega),\,\, \|u\|_E \leqslant 1\big \}.\] 
Similarly, if $f\in H^\infty(\Omega)^{(k)}:= H^\infty(\Omega)\otimes \mathbb{C}^{k\times k}$, then $\|f\|_{\Omega, \infty}^{(k)}:=\sup\{\|f(z)\|_{\ell_2^k \to \ell_2^k}\mid z\in \Omega\}$. Finally, note that there is a unique $C^*$- norm on $\mathbb{C}^{n \times n } \otimes \mathbb{C}^{k\times k}\cong \mathbb{C}^{n k \times n k}$. Now, consider the homomorphism 
\[\varrho_V^{(k)}:=\varrho_V\otimes I_k:  H^\infty(\Omega)^{(k)} \to \mathbb{C}^{n k \times n k}\] 
and the linear map 
\[L^{(k)}_V:= L_V\otimes I_k: E^{(k)} \to \mathbb{C}^{n_1 k \times n_2 k}.\] 
The homomorphism  $\varrho_V$ (resp., the linear map $L_V$) is said to be completely contractive if $\|\varrho^{(k)}_V\| \leqslant 1$ (resp., $\|L_V^{(k)}\| \leqslant 1$) for $k=1,2, \ldots $. It is shown in \cite[Lemma 3.3]{Misra94} (also, see \cite[Lemma 5.1]{vip}) that 
$\|\varrho_V^{(k)}\| \leqslant 1$ if and only if $\|L_V^{(k)}\| \leqslant 1$.  Therefore, $\varrho_V$ is completely contractive if and only if so is $L_V$. Thus, the question of determining when a contractive homomorphism of the form $\varrho_V$ is completely contractive is equivalent to answering the same question for the 
linear map $L_V$. We  say $\varrho_V$ is completely bounded if $\sup_{k\geq 1}\|\varrho_V^{(k)}\|<\infty$. In this case we denote $\|\varrho\|_{cb}=\sup_{k\geq 1}\|\varrho_V^{(k)}\|<\infty.$ Moreover, $\|\varrho\|_{cb}=\sup_{k\geq 1}\|L_V^{(k)}\|.$

 It is then natural to ask if contractive linear maps of the form $L_V$ are completely contractive, see \cite{vip} for more on this. This question is connected with the computation of the positive Grothendieck constant if we assume that all rows of $V_i$, $i=1, \ldots , m$, are zero except the first one. With a slight abuse of language, setting $n_1=1$ and $n_2=n$, we identify $V_i$ with $v_i \in \mathbb{C}^{1\times n}$.  For such a choice of the matrices $V_1, \ldots, V_m$, the homomorphism $\varrho_V(f)$ is of the form
 \[\varrho_V(f)= \begin{pmatrix} f(w)  & Df(w) \cdot V \\ 0 & f(w)I_{n} \end{pmatrix}.\]
 Here, the rows of the matrix $V$ are the vectors $v_1, \ldots , v_m$ and $Df(w) \cdot V$ is the $1\times n$ matrix obtained from the left action of $V$ on the derivative $Df(w)$. 
 We define \textit{little Parrott homomorphisms} to be the class of all the homomorphisms induced by these $1\times n$ matrices.  With these standing assumptions, Theorem 1.9 of \cite{BM} shows that a Banach space $E$ has Property P if and only if every contractive linear map $L_V: (\mathbb{C}^m, \|\cdot\|_{E^*}) \to \ell_2^n$ induced by a little Parrott homomorphism $\varrho_V$ is completely contractive. The $m$- tuple $V$ defining the little Parrott homomorphism can be expressed in terms of the curvature of a commuting tuple of operators in the Cowen-Douglas class $B_1(\Omega)$. We refer to \cite{Misra901, Misra94, Misra19,  Misra902, vip, RKSBLMS} for more on this topic. The following lemma is a slight variation of  \cite[Theorem 1.9]{BM}. 

\begin{theorem}\label{thm5.17}
Let \( V:=(V_1,\dots, V_m) \) be as above such that $\varrho_V$ is a contractive little Parrott homomorphism. Then 
\( \|\varrho_V\|_{cb}\leq\sqrt{\gamma(\mathbb{C}^m,\|\cdot\|_{\Omega} )}.\) Furthermore, the upper bound is attained by some contractive little Parrott homomorphism.

\end{theorem}

\begin{proof}
Let us denote \[\mathcal{B}_k:=\{B=(B_1,\dots,B_m):\left\| z_1 B_1 + \cdots + z_m B_m \right\|_{\text{op}} \leq 1,\ \forall z \in \Omega,\ B_i\in M_k, 1\leq i\leq m\}.\]
Note that by the discussion preceding the theorem, we clearly have
\begin{equation}\label{eqn5.2}
\|\varrho_V^{(k)}\|^2=\sup\{\left\| B_1 \otimes v_1 + \cdots + B_m \otimes v_m \right\|_{\text{op}}^2:(B_1,\dots,B_m)\in\mathcal{B}_k\}.
\end{equation}
 However, we observe that 
 \begin{equation}\label{eqn5.3}\big \| B_1 \otimes v_1 + \cdots + B_m \otimes v_m \big \|_{\text{op}}^2=\sup \Big \{ \Big | \sum_{i,j} \langle B_jx, B_ix \rangle \langle v_j, v_i \rangle \Big |: x \in \ell_2^k, \|x\|_2 \leq 1 \Big \}. \end{equation} 
 To this end, notice that 
 \[\{(B_1x,\dots,B_mx):x\in \ell_2^k,\ \|x\|_2\leq 1\}=\Big \{(b_1,\dots,b_m):\sup_{z\in\Omega}\Big\|\sum_{i=1}^mz_ib_i\Big\|_2\leq 1,\ b_i\in \ell_2(k)\Big \}.\] Therefore, by equations \eqref{eqn5.2} and \eqref{eqn5.3} we obtain 
\[\|\varrho_V^{(k)}\|^2=\sup\Big \{\Big| \sum_{i,j} \langle b_j, b_i \rangle \langle v_j, v_i \rangle \Big |:\sup_{z\in\Omega}\Big\|\sum_{i=1}^mz_ib_i\Big\|_2\leq 1,\ b_i\in \ell_2^k\Big \}.\] Moreover, $\sup_{z\in\Omega}\Big\|\sum_{i=1}^mz_ib_i\Big\|_2\leq 1$ if and only if $(\langle b_j,b_i\rangle)_{i,j=1}^m:(\mathbb{C}^m,\|.\|_{\Omega})\to (\mathbb{C}^m,\|.\|_{\Omega})^*$ is contractive. Also contractivity of $\varrho_V$ is equivalent to $(\langle v_j,v_i\rangle)_{i,j=1}^m:(\mathbb{C}^m,\|.\|_{\Omega})^*\to (\mathbb{C}^m,\|.\|_{\Omega})$ is contractive. Therefore, we readily conclude that 
\( \|\varrho_V\|_{cb}\leq\sqrt{\gamma(\mathbb{C}^m,\|\cdot\|_{\Omega} )}.\)) The fact that equality is achieved for some $V$ easily follows from a compactness argument. This completes the proof.
\end{proof} 
\subsection{$\gamma(E)$ for complex Banach spaces} In this subsection, all the Banach spaces are assumed to be over the field of complex numbers. Let $K_G^{+}(\mathbb C)$ be the complex positive Grothendieck constant. Theorem \ref{thm5.17} relates the completely bounded norms of little Parrott homomorphisms to $\gamma(E)$. In this subsection, we  find asymptotic behavior of ( also see Proposition \ref{++++}),) $\gamma(E)$  for various classical finite-dimensional Banach spaces. 
The constant $\alpha$ introduced in \cite{vip} is related to the Parrott homomorphisms exactly in the same way that the constant $\gamma(E)$ is related to the little Parrott homomorphisms. 
However, unlike $\alpha,$ here we find an exact asymptotic for $\gamma(\ell_p^n)$. Moreover, as shown below, $\gamma(\ell^3_\infty) = \sup \{K_G^{+}(\ell_\infty^3,\ell_2^m): m\in \mathbb{N}\} = 1$ while the Parrott example shows that $\alpha > 1$ for $\ell_\infty^3$.  However, the well-known theorem due to Ando showing that a pair of commuting contraction admits a unitary dilation means that $\alpha=1$ for $\ell_\infty^2$ and consequently $\gamma$ of this space is also $1$. Therefore, to determine if the two constants $\alpha$ and $\gamma$ are distinct for a finite dimensional Banach space appears to be an interesting problem.
\begin{proposition} \label{prop5.12}
 For $1\leqslant p\leqslant 2$, we have $n^{\frac{2}{p^\prime}}\leqslant\gamma(\ell_p^n)\leqslant K_G^{+}(\mathbb C)n^{\frac{2}{p^\prime}}.$
\end{proposition}
\begin{proof}
Since $p\leqslant p^\prime,$ we have that $\|x\|_{p^\prime}\leqslant\|x\|_p.$ Thus $\|\operatorname{id}\|_{\ell_p^n\to\ell_{p^\prime}^n}\leqslant 1.$ On the other hand we have \[\|\operatorname{id}\|_{\ell_{p^\prime}^n\to\ell_p^n}=n^{1-\frac{2}{p^\prime}}.\] Therefore, we have that \[\gamma(\ell_p^n)\geqslant \frac{\langle \text{id},\text{id}\rangle }{\|\text{id}\|_{\ell_{p^\prime}^n\to\ell_p^n}\|\operatorname{id}\|_{\ell_{p}^n\to\ell_{p^\prime}^n}}=n^{\frac{2}{p^\prime}}.\]  noting that $\gamma(\ell_p^n)\leqslant d(\ell_p^n,\ell_1^n)^2\gamma(\ell_1^n)\leqslant n^{\frac{2}{p^\prime}}K_G^{+}(\mathbb C)$ and using the known bounds of $d(\ell_1^n,\ell_p^n)$, see \cite[Proposition 37.6]{Tom89}. 
\end{proof}
\begin{remark} 
We recall from \cite{BM} that $\gamma(\ell_1^2) = 1$. Consequently, taking 
$1 \leqslant p \leqslant 2$ and $n=2$, we have $\gamma(\ell_p^2) = 2^{\tfrac{2}{p^\prime}},$ as $d(\ell_1^2,\ell_p^2)= 2^{\tfrac{2}{p^\prime}}$. From the definition of $\gamma$, it follows that 
$\gamma(E) = \gamma(E^*)$. As a result, for $2 \leqslant p \leqslant \infty$, we have $\gamma(\ell_p^2) = 2^{\tfrac{2}{p}}$.
\end{remark}
\begin{proposition}\label{prop5.13}
We have the following estimates for Schatten-$p$ classes.
\begin{itemize}
    \item[(i)] $n\leqslant \gamma(S_1^n).$
    \item[(ii)] Let $1<p<2.$ Then $n^{1+\frac{2}{p^\prime}}\leqslant \gamma(S_p^n).$
\end{itemize}
\end{proposition}

\begin{proof}
Note that $\operatorname{id}:S_\infty^n\to S_1^n$ has operator norm $n$ and $\operatorname{id}:S_1^n\to S_\infty^n$ has norm $1.$ Therefore, we have that
\[\gamma(S_1^n)\geqslant \frac{\langle \operatorname{id},\operatorname{id}\rangle }{\|\operatorname{id}\|_{S_{1}^n\to S_\infty^n}\|\operatorname{id}\|_{S_{\infty}^n\to S_{1}^n}}=n.\] Thus $\gamma(S_1^n)\geqslant {n}.$   Part (ii) follows from a similar calculation.
\end{proof}
\begin{remark}
Propositions  \ref{prop5.12} and \ref{prop5.13} show the difference between $\ell_p^n$ and $S_p^n$ through the distinct behavior of the constants $\gamma(\ell_p^n)$ and $\gamma(S_p^n)$.
\end{remark}
Now, we give elementary and short verification of the equalities  
\[\sup_{m\geqslant 1}K_G^{+}(\ell_\infty^2,\ell_2^m)=\sup_{m\geqslant 1}K_G^{+}(\ell_\infty^3,\ell_2^m)=1.\] 
These were proved earlier in \cite[Theorem 4.2]{AFHS95}, see Proposition \ref{++++} (and also  \cite[Fact 7]{BM} together with \cite[Fact 2]{BM}). 
 In \cite{BM08},  the Grothendieck constants in dimensions $2$ and $3$ were calculated using the known list of extreme points (of an appropriate unit ball) in the real case. The computation in the complex case used different techniques. Here, we use an upper bound on the rank of extreme points of correlation matrices given in \cite{LCB} in the complex case to show that the positive Grothendieck constant must be $1$ for $n=2$ or $3$. The proof is from \cite{RKS}. 

Let $E$ be a finite-dimensional Banach space, we say $E$ has $2$-summing property if for any Hilbert space $\mathcal{H}$ and $T\in B(E,\mathcal{H}),$ $\pi_2(T):=\|T\|$.  
Using Proposition \ref{++++}, we get that  \begin{equation}\label{kgpi}K_G^+(\ell_\infty^n,\ell_2^k)=\sup\{\pi_2(Tu)\}:\|u\|_{\ell_2^k\to\ell_\infty^n}\leqslant 1,\|T\|_{\ell_\infty^n\to{\ell_2^k}}\leqslant 1\}.\end{equation} 

We recall below a useful Lemma from \cite{Rietz}.

\begin{lemma}\label{noteasy} Suppose that $A=(a_{ij})_{i,j=1}^n\in M_n$  is a  non-negative matrix. Then we have 
\[\sup\Big\{\Big|\sum_{i,j=1}^na_{ij}\langle x_i,x_j\rangle\Big|:\|x_i\|_{\mathcal{H}}\leqslant 1\Big\}=\sup\Big\{\Big|\sum_{i,j=1}^na_{ij}\langle x_i,x_j\rangle\Big|:\|x_i\|_{\mathcal{H}}= 1\Big\},\label{(i)}\tag{$\dagger$}\]
for any Hilbert space $\mathcal H,$ and 
\[\|A\|_{\ell_\infty^n\to\ell_1^n}=\sup_{\begin{array}{c}|z_i|=1\\ 1\leqslant i\leqslant n\end{array}} 
\sum_{i,j=1}^n a_{ij}z_i\overline{z}_j.\label{(ii)}\tag{$\ddagger$}\] 
\end{lemma}

In \cite{AFHS95}, the authors constructed an operator $T:\ell_\infty^4\to \ell_2^2$ such that $\pi_2(T)>\|T\|.$ This, in turn, shows that $K^+_G(\ell_\infty^4)>1$.
However, they do not mention any explicit lower bound.  Here, following their methods, we find an explicit lower bound of $K_G^{+}(\ell_\infty^4,\ell_2^2).$ From this, we find an alternative proof of \cite[Theorem 2.1] {Davidchoi}  as well as explicit examples. 

Part (i) of the theorem below has been proved in \cite{AFHS95} and also in \cite{BM}. 
Recall that a complex positive semi-definite matrix with all its diagonal elements equal to $1$ is called a \textit{correlation matrix}. We denote the set of all $n\times n$ correlation matrices by $\mathcal{C}(n)$.

\begin{theorem}\label{plasth}
We have the following.
\begin{itemize}
    \item[(i)] For all $m\in\mathbb N,$ $K_G^{+}(\ell_\infty^2,\ell_2^m)=K_G^{+}(\ell_\infty^3,\ell_2^m)=1.$
    \item[(ii)] $K_G^{+}(\ell_\infty^4,\ell_2^2)\geqslant 1.1658.$
\end{itemize}
 \end{theorem}
\begin{proof}
Given a complex $n\times n$ non-negative matrix $A$, set  
\[\beta(A)=\sup\big\{ \langle A,B\rangle\mid B\geqslant 0, \|B\|_{\ell_1^n\to\ell_\infty^n}\leqslant 1 \big \}.\] 
Note that part (i) of Lemma \ref{noteasy} taken together with Fact \ref{fact6.7} shows that 
\[\beta(A)=\sup\{\langle A,B\rangle: B\in\mathcal{C}(n)\}.\] 
Since $\langle A,B\rangle$ is $\mathbb R$-linear in $B$ and $\mathcal{C}(n)$ is a compact convex set, therefore it follows that 
\[\beta(A)=\sup\{\langle A,B\rangle: B\in E(\mathcal{C}(n))\},\] 
where $E(\mathcal{C}(n))$ is the set of all extreme points of $\mathcal{C}(n)$. Since, all the elements of $E(\mathcal{C}(n))$ are of rank less than or equal to $\sqrt{n}$, see \cite{LCB}, in case $n$ is either $2$ or $3$, we conclude that extreme correlation matrices are of rank one. Now, if the correlation matrix $B=(\langle x_i,x_j\rangle)_{i,j=1}^n$ is of rank $1$, then $x_i$'s can be chosen to be one dimensional unit vectors. So for $n=2,3$, we obtain the following \[\beta(A)=\sup_{B\in E(\mathcal{C}(n))}\langle A,B\rangle=\sup_{\mid z_i\mid=1}\sum_{i,j=1}^{n}a_{ij}z_i\bar{z}_{j}=\|{A} \|_{\ell_\infty^n\to\ell_1^n}.\] The last equality follows from part (ii) of Lemma \ref{noteasy}
completing the proof of part (i) of the theorem.

Let \( L_p^n(\mu) \) denote the linear space of functions \( f: \{1, \dots, n\} \to \mathbb{C} \), equipped with the usual \( L_p \)-norm, where \( \mu \) is the uniform probability measure on \( \{1, \dots, n\} \).
Let 
\[x_1=\Big(1,0,\frac{1}{\sqrt{2}},\frac{1}{\sqrt{2}}\Big) \text{~and~} x_2=\Big(1,0,\frac{i}{\sqrt{2}},\frac{1}{\sqrt{2}}\Big)\in L_\infty^4(\mu)\] and consider $X=\text{span}\{x_1,x_2\}.$ Also, we let  $X_\infty$ and $X_2$ to be the linear space $X$ considered as subspaces of $L_\infty^4(\mu)$ and $L_2^4(\mu)$, respectively. 
Let $I_{\infty,2}^X$ denote the restriction of the identity map $I_{\infty,2}:X_\infty\to X_2.$ By \cite[Example 2.3]{AFHS95}, we must have that $\|I_{\infty,2}^X\|<1$ but $\pi_2(I_{\infty,2}^X)=1.$ Moreover, if we define $u:=PI_{\infty,2}:L_\infty^4(\mu)\to X_2$ as $P:L_2^4(\mu)\to X_2$ is an orthogonal projection, by \cite{AFHS95} again, $\|u\|<1$ and $\pi_2(u)=1.$ We'll explicitly compute $\|u\|.$ To this end, note that $${X_2}^{\perp}=\text{span}\left\{(0,1,0,0),\big(-\frac{1}{\sqrt{2}},0,0,1\big)\right\}.$$ 
Let $\alpha,\beta,\delta,\gamma\in\mathbb{C}.$ Since $P$ is an orthogonal projection onto $X_2,$ therefore,  we get, 
\[P\left(\Big(\alpha+\beta-\frac{\delta}{\sqrt{2}},\gamma,\frac{\alpha+i\beta}{\sqrt{2}},\frac{\alpha+\beta}{\sqrt{2}}+\delta\Big)\right)=\Big(\alpha+\beta,0,\frac{\alpha+i\beta}{\sqrt{2}},\frac{\alpha+\beta}{\sqrt{2}}\Big)=(t_1,t_2,t_3,t_4) \quad (say).\]
 In these notations, we obtain $$\delta=\frac{2}{3}\Big(t_4-\frac{t_1}{\sqrt{2}}\Big), \quad \alpha+\beta=t_1+\frac{\delta}{\sqrt{2}}=t_1+\frac{\sqrt{2}}{3}\Big(t_4-\frac{t_1}{\sqrt{2}}\Big).$$  
 Putting above computations together, we get 
\[P(t_1,t_2,t_3,t_4)=\Big(\frac{2}{3}t_1+\frac{\sqrt{2}}{3}t_4,0,t_3,\frac{\sqrt{2}}{3}t_1+\frac{1}{3}t_4\Big).\] This, in turn gives, 
\[u(t_1,t_2,t_3,t_4)=\Big(\frac{2}{3}t_1+\frac{\sqrt{2}}{3}t_4,0,t_3,\frac{\sqrt{2}}{3}t_1+\frac{1}{3}t_4\Big).\] Hence, we get, 
\[\|u\|^2=\sup\left\{\frac{1}{4}\Big|\frac{2}{3}t_1+\frac{\sqrt{2}}{3}t_4\Big|^2+\frac{1}{4}|t_3|^2+\frac{1}{4}\Big|\frac{\sqrt{2}}{3}t_1+\frac{1}{3}t_4\Big|^2:|t_i|\leqslant 1,\ 1\leqslant i\leqslant 4\right\}.\] 
Clearly  the supremum above  is achieved at $t_1=t_2=t_3=t_4=1$. Thus we obtain $\|u\|= \sqrt{\frac{3+\sqrt{2}}{6}}\approx\sqrt{.7357}.$ Hence $K_G^{+}(\ell_\infty^4,\ell_2^2)\geqslant \frac{\pi_2(u)}{\|u\|}\approx 1.1658.$
\end{proof}

\begin{corollary}[Theorem 2.1, \cite{Davidchoi}]\label{choicor} 
There exists a quadruple of $3\times 3$ commuting tuple of matrices which are contractions but they do not coextend to commuting isometries.
\end{corollary}
\begin{proof}
    Note that in view of Theorem \ref{thm5.17} and Part (ii) of Theorem \ref{plasth}, there exists a contractive commuting tuple of the form 
    \[\begin{pmatrix}
 0 & a_1 & a_2 \\
 0 & 0 & 0 \\
 0 & 0 & 0
\end{pmatrix}, \begin{pmatrix}
 0 & b_1 & b_2 \\
 0 & 0 & 0 \\
 0 & 0 & 0
\end{pmatrix}, \begin{pmatrix}
 0 & c_1 & c_2 \\
 0 & 0 & 0 \\
 0 & 0 & 0
\end{pmatrix} \text{and}\ \begin{pmatrix}
 0 & d_1 & d_2 \\
 0 & 0 & 0 \\
 0 & 0 & 0
\end{pmatrix}\] which does not satisfy matrix-valued von Neumann inequality, or equivalently, do not extend to commuting isometries but satisfies the von Neumann inequality. 
Thus, the proof is complete. 
\end{proof}
\begin{remark}
We have the following remarks.
\begin{itemize}
    \item[(i)] In the above corollary, one can explicitly find the vectors $(a_1,a_2),$ $(b_1,b_2),$ $(c_1,c_2)$ and $(d_1,d_2).$   
    \item[(ii)] Since $\sup_{n\geq 1}\gamma(\ell_\infty^n)<\infty$, it follows that the completely bounded norms of any contractive little Parrott homomorphism defined on $\ell_\infty^n, n\in \mathbb{N}$, is bounded by an universal constant independent of $n$. In view of \ref{xample}, Theorem \ref{thm5.17} and Proposition \ref{++++} we can find sequence of domains $(\Omega_n)_{n\geq 1}\subseteq\mathbb{C}^n$ which are unit balls with respect to some norms such that the natural inclusion $(\mathbb{C}^n,\|.\|_{\Omega_n})\hookrightarrow (\mathbb{C}^{n+1},\|.\|_{\Omega_{n+1}})$ is an isometry. Moreover, for each \(n\), the completely bounded norm of any contractive little Parrott homomorphism defined on \(\Omega_n\) remains bounded by a universal constant. (See also Remark~\ref{nxampl}.)
\end{itemize}
   
\end{remark}

\subsection*{Acknowledgment} We are very grateful to G. Pisier for generously sharing his ideas on the topic of this paper. We also thank him for pointing out an anomaly in an earlier draft of the paper. The authors thank Md. Ramiz Reza for some discussion on Corollary \ref{choicor}.

\subsection*{Funding sources} 
The first-named author acknowledges support from the ANRF ARG grant \texttt{ANRF/ARG/2025/001228/MS}.
The third named author thanks the DST-INSPIRE Faculty Fellowship
\texttt{DST/INSPIRE/04/2020/001132} and Prime Minister Early Career Research Grant Scheme \texttt{ANRF/ECRG/2024/000699/PMS} and \texttt{ANRF/ARGM/2025/000895/MTR}.

\end{document}